\documentclass[12pt,reqno]{amsart}

\usepackage[all]{xy}
\usepackage{tom}
\DeclareMathOperator{\syz}{syz}
\DeclareMathOperator{\IN}{in}
\DeclareMathOperator{\tin}{t-in}

\DeclareMathOperator{\val}{val}
\DeclareMathOperator{\spoly}{spoly}

\DeclareMathOperator{\conv}{conv}
\DeclareMathOperator{\lm}{lm}
\DeclareMathOperator{\lt}{lt}
\DeclareMathOperator{\lc}{lc}
\DeclareMathOperator{\TAIL}{tail}
\DeclareMathOperator{\HDDwR}{HDDwR}
\DeclareMathOperator{\DwR}{DwR}
\DeclareMathOperator{\STD}{STD}

\DeclareMathOperator{\INTERSECTION}{INTERSECTION}
\DeclareMathOperator{\ELIMINATE}{ELIMINATE}
\DeclareMathOperator{\QUOTIENT}{QUOTIENT}
\DeclareMathOperator{\SATURATION}{SATURATION}
\DeclareMathOperator{\MEMBERSHIP}{MEMBERSHIP}
\DeclareMathOperator{\ecart}{ecart}

\newcommand{\ux}{\underline{x}}
\newcommand{\ut}{\underline{t}}
\newcommand{\uz}{\underline{z}}

\newcommand{\lang}[1]{}
\newcommand{\kurz}[1]{#1}
\newcommand{\bmath}{\kurz{\begin{math}}\lang{\begin{displaymath}} }
\newcommand{\emath}{\kurz{\end{math}}\lang{\end{displaymath}} }

\begin{document}

   \parindent0cm

   \title[Standard Bases]{Standard Bases in $K[[t_1,\ldots,t_m]][x_1,\ldots,x_n]^s$}
   \author{Thomas Markwig}
   \address{Universit\"at Kaiserslautern\\
     Fachbereich Mathematik\\
     Erwin--Schr\"odinger--Stra\ss e\\
     D --- 67663 Kaiserslautern\\
     Tel. +496312052732\\
     Fax +496312054795
     }
   \email{keilen@mathematik.uni-kl.de}
   \urladdr{http://www.mathematik.uni-kl.de/\textasciitilde keilen}
   \thanks{The author was supported by the IMA, Minneapolis.}


   \date{March, 2007.}

     
   \begin{abstract}
     In this paper we study standard bases for submodules of
     $K[[t_1,\ldots,t_m]][x_1,\ldots,x_n]^s$ respectively of their
     localisation with respect to a $\ut$-local
     monomial ordering. The main step is to prove the existence of
     a division with remainder generalising and combining the
     division theorems of Grauert and Mora. Everything else then
     translates naturally. Setting either $m=0$ or $n=0$ we get
     standard bases for polynomial rings respectively for power series
     rings as a special case. We then apply this technique to show
     that the $t$-initial ideal of an ideal over the Puiseux series
     field can be read of  from a standard basis of its
     generators. This is an important step in the constructive proof
     that each point in the tropical variety of such an ideal admits a
     lifting.
   \end{abstract}

   \maketitle

   The paper follows the lines of \cite{GP02} and \cite{DS07}
   generalising the results where necessary. Basically, the only
   original parts for the standard bases are the proofs of Theorem
   \ref{thm:HDDwR} and 
   Theorem \ref{thm:DwR}, but even here they are easy generalisations
   of Grauert's respectively Mora's Division Theorem (the latter in the form
   stated and proved first by Greuel and Pfister, see \cite{GP96}; see
   also \cite{Gra94}). The paper should
   therefore rather be seen as a unified approach for the existence
   of standard bases in polynomial and power series rings, and it was
   written mostly due to the lack of a suitable reference for the
   existence of standard bases in $K[[t]][x_1,\ldots,x_n]$ which are
   needed when dealing with tropical varieties. Namely, when we want
   to show that every point in the tropical variety of an ideal $J$
   defined over the field of Puiseux series exhibits a lifting to the
   variety of $J$, then, assuming that $J$ is generated by elements in 
   $K\big[\big[t^\frac{1}{N}\big]\big][x_1,\ldots,x_n]$, we need to
   know that we can compute the so-called $t$-initial ideal of $J$ by
   computing a standard basis of the ideal defined by the  generators
   in $K\big[\big[t^\frac{1}{N}\big]\big][x_1,\ldots,x_n]$ (see Theorem
   \ref{thm:stdtin} and \cite{JMM07}).

   An important point is that if the input data is
   polynomial in both $\ut$ and $\ux$ then we can actually compute the
   standard basis since a standard basis computed in $K[t_1,\ldots,t_m]_{\langle
     t_1,\ldots,t_m\rangle}[x_1,\ldots,x_n]$ will do (see Corollary
   \ref{cor:polynomialcase}). This was previously known 
   for the case where there are no $x_i$ (see \cite{GP96}).

   In Section \ref{sec:basicnotion} we introduce the basic notions.
   Section \ref{sec:HDDwR} is devoted to the proof of the existence of
   a determinate division with remainder for polynomials
   in $K[[t_1,\ldots,t_m]][x_1,\ldots,x_m]^s$ which are homogeneous with
   respect to the $x_i$. This result is then used in Section
   \ref{sec:DwR} to show the existence of weak divisions with
   remainder for all elements of
   $K[[t_1,\ldots,t_m]][x_1,\ldots,x_m]^s$. In Section
   \ref{sec:standardbases} we introduce standard bases and prove
   the basics for these, and we prove Schreyer's Theorem and, thus
   Buchberger's Criterion in  Section \ref{sec:schreyer}. 
   \lang{In Section \ref{sec:algorithms} we describe some algorithms
     which rely on the standard basis algorithm, and if the input is
     polynomial in $\ut$ as well as in $\ux$ then the algorithms
     terminate.}
   Finally,
   in Section \ref{sec:application} we apply standard bases to study
   $t$-initial ideals of ideals over the Puiseux series field.


   \section{Basic Notation}\label{sec:basicnotion}

   Throughout the paper $K$ will be any field, $R=K[[t_1,\ldots,t_m]]$ will
   denote the ring of formal power series over $K$ and 
   \begin{displaymath}
     R[x_1,\ldots,x_n]=K[[t_1,\ldots,t_m]][x_1,\ldots,x_n]
   \end{displaymath}
   denotes the ring of polynomials in the indeterminates
   $x_1,\ldots,x_n$ with coefficients in the power series ring
   $R$. We will in general use the
   short hand notation $\ux=(x_1,\ldots,x_n)$ and
   $\ut=(t_1,\ldots,t_m)$, and the usual multi index notation 
   \begin{displaymath}
     \ut^\alpha=t_1^{\alpha_1}\cdots t_m^{\alpha_m}\;\;\;\mbox{ and
     }\;\;\; \ux^\beta=x_1^{\beta_1}\cdots x_n^{\beta_n},
   \end{displaymath}
   for $\alpha=(\alpha_1,\ldots,\alpha_m)\in\N^m$ and
   $\beta=(\beta_1,\ldots,\beta_n)\in\N^n$ . 

   \begin{definition}
     A \emph{monomial ordering} on 
     \begin{displaymath}
       \Mon(\ut,\ux)=\big\{\ut^\alpha\cdot\ux^\beta\;\big|\;\alpha
       \in\N^m,\beta\in\N^n\big\}
     \end{displaymath}
     is a total ordering $>$ on $\Mon(\ut,\ux)$ which is compatible
     with the semi group structure of $\Mon(\ut,\ux)$, i.e.\ such that
     for all $\alpha,\alpha',\alpha''\in\N^m$ and $\beta,\beta',\beta''\in\N^n$
     \begin{displaymath}
       \ut^\alpha\cdot \ux^\beta\;>\;\ut^{\alpha'}\cdot\ux^{\beta'}
       \;\;\;\Longrightarrow\;\;\;
       \ut^{\alpha+\alpha''}\cdot \ux^{\beta+\beta''}
       \;>\;\ut^{\alpha'+\alpha''}\cdot\ux^{\beta'+\beta''}.
     \end{displaymath}
     We call a monomial ordering $>$ on $\Mon(\ut,\ux)$
     \emph{$\ut$-local} if its restriction to $\Mon(\ut)$
     is \emph{local}, i.e.
     \bmath
       t_i<1\lang{\;\;\;}\mbox{ for all }\lang{\;}i=1,\ldots,m.
     \emath
     We call a $\ut$-local monomial ordering on $\Mon(\ut,\ux)$ a
     \emph{$\ut$-local weighted degree ordering} if
     there is a $w=(w_1,\ldots,w_{m+n})\in \R_{\leq 0}^m\times\R^n$ 
     such that for all $\alpha,\alpha'\in \N^m$ and $\beta,\beta'\in\N^n$
     \begin{displaymath}
       w\cdot (\alpha,\beta)>w\cdot
       (\alpha',\beta')\;\;\;\Longrightarrow\;\;\;\ut^\alpha\cdot\ux^\beta>\ut^{\alpha'}\cdot\ux^{\beta'},
     \end{displaymath}
     where
     $w\cdot(\alpha,\beta)=w_1\cdot\alpha_1+\ldots+w_m\cdot\alpha_m+w_{m+1}\cdot\beta_1+\ldots 
     +w_n\cdot\beta_n$ denotes the standard scalar product. We
     call $w$ a \emph{weight vector} of $>$. 
   \end{definition}

   \begin{example}\label{ex:lex}
     The \emph{$\ut$-local lexicographical ordering} $>_{lex}$ on
     $\Mon(\ut,\ux)$ is defined by
     \begin{displaymath}
       \ut^\alpha\cdot\ux^\beta\;>\;\ut^{\alpha'}\cdot\ux^{\beta'}
     \end{displaymath}
     if and only if
     \begin{displaymath}
       \exists\;j\in\{1,\ldots,n\}\;:\;\beta_1=\beta_1',\ldots,\beta_{j-1}=\beta_{j-1}',
       \;\mbox{ and }\;\beta_j>\beta_j',
     \end{displaymath}
     or
     \begin{displaymath}
       \big(\beta=\beta'\mbox{ and }
       \exists\;j\in\{1,\ldots,m\}\;:\;\alpha_1=\alpha_1',\ldots,\alpha_{j-1}=\alpha_{j-1}',
       \alpha_j<\alpha_j'\big).
     \end{displaymath}
   \end{example}

   \begin{example}\label{ex:weighted}
     Let $>$ be any $\ut$-local ordering and
     $w=(w_1,\ldots,w_{m+n})\in\R_{\leq 0}^m\times\R^n$, then 
     \bmath
       \ut^\alpha\cdot\ux^\beta>_w\ut^{\alpha'}\cdot\ux^{\beta'}
     \emath
     if and only if 
     \bmath
       w\cdot(\alpha,\beta)>w\cdot(\alpha',\beta')
     \emath
     or 
     \begin{displaymath}
       \big(w\cdot(\alpha,\beta)=w\cdot(\alpha',\beta')\;\mbox{ and }\;
       \ut^\alpha\cdot\ux^\beta>\ut^{\alpha'}\cdot\ux^{\beta'}\big)
     \end{displaymath}
     defines a $\ut$-local weighted degree ordering $>_w$ on
     $\Mon(\ut,\ux)$ with weight vector $w$.
   \end{example}

   Even if we are only interested in standard bases of ideals we have
   to pass to submodules of free modules in order to have syzygies at
   hand for the proof of Buchberger's Criterion via Schreyer
   orderings. 

   \begin{definition}
     We define
     \begin{displaymath}
       \Mon^s(\ut,\ux):=\big\{\ut^\alpha\cdot\ux^\beta\cdot
       e_i\;|\;\alpha\in\N^n,\beta\in\N^m, i=1,\ldots,s\big\},
     \end{displaymath}
     where $e_i=(\delta_{ij})_{j=1,\ldots,s}$ is the vector with all
     entries zero except the $i$-th one which is one. We call the
     elements of $\Mon^s(\ut,\ux)$ \emph{module monomials} or simply
     \emph{monomials}. 

     For $p,p'\in\Mon^s(\ut,\ux)\cup\{0\}$ 
     \kurz{ the notion of divisibility and of the lowest common
       multiple $\lcm(p,p')$ are defined in the obvious way. }
     \lang{
     we define
     \begin{displaymath}
       p\;\big|\;p'
     \end{displaymath}
     in words ``$p$ divides $p'$'',
     if and only if
     \begin{displaymath}
       \;\;\exists\;\alpha\in\N^m,\beta\in\N^n\;:\;
       \ut^\alpha\cdot\ux^\beta\cdot p=p'.
     \end{displaymath}
     If $p\not=0$, then we define in this case 
     \begin{displaymath}
       \frac{p}{p'}
       :=\ut^{\alpha''}\cdot\ux^{\beta''}\in\Mon(\ut,\ux).
     \end{displaymath}
     Note that, this is well-defined since $\beta$ and $\alpha$ are
     uniquely determined if $p\not=0$, and note also that for
     $p=\ut^\alpha\cdot\ux^\beta\cdot e_i$ and
     $p'=\ut^{\alpha'}\cdot\ux^{\beta'}\cdot e_j$ the condition
     $p\;|\;p'$ necessarily implies that $i=j$.

     Moreover, given two monomials $\ut^\alpha\cdot\ux^\beta\cdot e_i,
     \ut^{\alpha'}\cdot\ux^{\beta'}\cdot e_j\in\Mon^s(\ut,\ux)$ we define the
     \emph{lowest common multiple} of the two as
     \begin{displaymath}
       \lcm\big(\ut^\alpha\cdot\ux^\beta\cdot e_i,
       \ut^{\alpha'}\cdot\ux^{\beta'}\cdot e_j\big)
       :=\left\{
         \begin{array}{ll}
           t_1^{\max(\alpha_1,\alpha_1')}\cdots
           x_n^{\max(\beta_n,\beta_n')}, &
           \mbox{ if } i=j,\\
           0&\mbox{ if } i\not=j.
         \end{array}
       \right. 
     \end{displaymath}    
     Thus the least common multiple of two monomials is somehow the \emph{smallest}
     monomial which is divisible by both monomials, in the sense that
     it does divide every other monomial with this property.
     }

     Given a monomial ordering on $\Mon(\ut,\ux)$,
     a \emph{$\ut$-local monomial ordering} on $\Mon^s(\ut,\ux)$ with 
     respect to $>$ is a total ordering $>_m$ on $\Mon^s(\ut,\ux)$ which is
     strongly compatible with the operation of the multiplicative
     semi group $\Mon(\ut,\ux)$ on $\Mon^s(\ut,\ux)$ in the sense that
     \begin{displaymath}
       \ut^\alpha\cdot\ux^\beta\cdot e_i\;>_m\;\ut^{\alpha'}\cdot\ux^{\beta'}\cdot e_j
       \;\;\;\Longrightarrow\;\;\;
       \ut^{\alpha+\alpha''}\cdot\ux^{\beta+\beta''}\cdot e_i\;>_m\;
       \ut^{\alpha'+\alpha''}\cdot\ux^{\beta'+\beta''}\cdot e_j
     \end{displaymath}
     and 
     \begin{displaymath}
       \ut^\alpha\cdot\ux^\beta\;>\;\ut^{\alpha'}\cdot\ux^{\beta'}
       \;\;\;\Longleftrightarrow\;\;\;
       \ut^\alpha\cdot\ux^\beta\cdot e_i\;>_m\;\ut^{\alpha'}\cdot\ux^{\beta'}\cdot e_i
     \end{displaymath}
     for all $\beta,\beta',\beta''\in\N^n$,
     $\alpha,\alpha',\alpha''\in\N^m,i,j\in \{1,\ldots,s\}$.

     Note that due to the second condition the ordering $>_m$ on $\Mon^s(\ut,\ux)$
     determines the ordering $>$ on $\Mon(\ut,\ux)$ uniquely, and we
     will therefore usually not distinguish between them, i.e.\ we
     will use the same notation $>$ also for $>_m$, and we will not
     specify the monomial ordering on $\Mon(\ut,\ux)$ in advance, but
     instead refer to it as the \emph{induced monomial ordering on
       $\Mon(\ut,\ux)$}. 

     We call a monomial ordering on $\Mon^s(\ut,\ux)$
     \emph{$\ut$-local} if the induced monomial ordering on
     $\Mon(\ut,\ux)$ is so. 

     We call a $\ut$-local monomial ordering on $\Mon^s(\ut,\ux)$ a
     \emph{$\ut$-local weight ordering} if
     there is a $w=(w_1,\ldots,w_{m+n+s})\in \R_{\leq 0}^m\times\R^n\times\R^s$ 
     such that for all $\alpha,\alpha'\in \N^m$, $\beta,\beta'\in\N^n$ and $i,j\in\{1,\ldots,s\}$
     \begin{displaymath}
       w\cdot (\alpha,\beta,e_i)>w\cdot
       (\alpha',\beta',e_j)\;\;\;\Longrightarrow\;\;\;
       \ut^\alpha\cdot\ux^\beta\cdot e_i>\ut^{\alpha'}\cdot\ux^{\beta'}\cdot e_j,
     \end{displaymath}
     and we call $w$ a \emph{weight vector} of $>$. 
   \end{definition}

   \lang{
   \begin{example}
     Let $>$ be a $\ut$-local monomial ordering on $\Mon(\ut,\ux)$.
     \begin{enumerate}
     \item We can extend $>$ to a $\ut$-local monomial ordering on
       $\Mon^s(\ut,\ux)$ in two straight forward ways by
       \begin{displaymath}
         \ut^\alpha\cdot\ux^\beta\cdot e_i\;>\;\ut^{\alpha'}\cdot\ux^{\beta'}\cdot e_j
       \end{displaymath}
       if and only if
       \begin{displaymath}
         i<j\;\;\mbox{ or }\;\;
         (i=j\;\mbox{ and }\; \ut^\alpha\cdot\ux^\beta>\ut^{\alpha'}\cdot\ux^{\beta'}),
       \end{displaymath}
       respectively by
       \begin{displaymath}
         \ut^\alpha\cdot\ux^\beta\cdot e_i\;>\;\ut^{\alpha'}\cdot\ux^{\beta'}\cdot e_j
         \end{displaymath}
         if and only if 
         \begin{displaymath}          
           \ut^\alpha\cdot\ux^\beta>\ut^{\alpha'}\cdot\ux^{\beta'}\;\;\mbox{ or }\;\;
         (\ut^\alpha\cdot\ux^\beta=\ut^{\alpha'}\cdot\ux^{\beta'}\;\mbox{ and }\; i<j),
       \end{displaymath}
       the first ordering giving \emph{priority to the components}, the second one
       giving \emph{priority to the monomials}. 
     \item        If $w\in\R_{\leq 0}^m\times\R^n$ and $>_w$ is the $\ut$-local weighted
       degree ordering with respect to $>$ and $w$ from Example
       \ref{ex:weighted} then the corresponding module monomial
       ordering giving priority to the  monomials is a $\ut$-local
       weight ordering on $\Mon^s(\ut,\ux)$ with weight vector
       $(w,0,\ldots,0)$. 
     \item The module monomial ordering corresponding to $>$ and
       giving priority to the components is a $\ut$-local weight ordering
       with respect to the weight vector
       $(0,\ldots,0,s,s-1,\ldots,1)$.
     \end{enumerate}
   \end{example}
   }

   \begin{example}\label{ex:weighteddegree}
     Let $w\in\R_{\leq 0}^m\times\R^{n+s}$ and let
     $>$ be any $\ut$-local monomial ordering on
     $\Mon^s(\ut,\ux)$ such that the induced $\ut$-local monomial
     ordering on $\Mon(\ut,\ux)$ is a $\ut$-local weighted degree
     ordering with respect to the weight vector $(w_1,\ldots,w_{m+n})$. Then
     \begin{displaymath}
       \ut^\alpha\cdot\ux^\beta\cdot e_i\;>_w\;\ut^{\alpha'}\cdot\ux^{\beta'}\cdot e_j
     \end{displaymath}
     if and only if
     \begin{displaymath}
       w\cdot(\alpha,\beta,e_i)>w\cdot(\alpha',\beta',e_j)
     \end{displaymath}
     or
     \begin{displaymath}
       \big(w\cdot(\alpha,\beta,e_i)=w\cdot(\alpha',\beta',e_j)\;\mbox{ and }\;
       \ut^\alpha\cdot\ux^\beta\cdot
       e_i>\ut^{\alpha'}\cdot\ux^{\beta'}\cdot e_j\big)         
     \end{displaymath}
     defines a $\ut$-local weight monomial ordering on
     $\Mon^s(\ut,\ux)$ with weight vector $w$. In particular, there
     exists such a monomial ordering.
   \end{example}
   
   \begin{remark}
     In the following we will mainly be concerned with monomial
     orderings on $\Mon^s(\ut,\ux)$ and with submodules of free modules over
     $R[\ux]$, but all these results specialise to $\Mon(\ut,\ux)$
     and ideals by just setting $s=1$.
     \hfill$\Box$
   \end{remark}

   \lang{
   The following lemma follows easily from the above
   definitions.

   \begin{lemma}
     The following conditions for a monomial ordering $>$ on
     $\Mon^s(\ut,\ux)$ are equivalent:
     \begin{enumerate}
     \item $>$ is $\ut$-local.
     \item $\ut^\alpha<1$ for all $0\not=\alpha\in\N^m$.
     \item $\ut^{\alpha+\alpha'}\cdot
       \ux^\beta<\ut^{\alpha'}\cdot\ux^\beta$ for all
       $\alpha,\alpha'\in\N^m$, $\beta\in\N^n$.
     \item $\ut^{\alpha+\alpha'}\cdot
       \ux^\beta\cdot e_i<\ut^{\alpha'}\cdot\ux^\beta\cdot e_i$ for all
       $\alpha,\alpha'\in\N^m$, $\beta\in\N^n$, $i=1,\ldots,s$.
     \end{enumerate}
   \end{lemma}
   }

   For a $\ut$-local monomial ordering we can introduce the notions of leading
   monomial and leading term of elements in \lang{$R[\ux]$.}\kurz{$R[\ux]^s$.}

   \lang{
   \begin{definition}
     Let $>$ be a $\ut$-local monomial ordering on $\Mon(\ut,\ux)$.
     We call
     \begin{displaymath}
       0\not=f=\sum_{|\beta|=0}^d \sum_{|\alpha|=0}^\infty
       a_{\alpha,\beta}\cdot \ut^\alpha\cdot \ux^\beta\in R[\ux],       
     \end{displaymath}     
     with
     $a_{\alpha,\beta}\in K$, $|\beta|=\beta_1+\ldots+\beta_n$
     and $|\alpha|=\alpha_1+\ldots+\alpha_m$, the \emph{distributive
       representation} of $f$, 
     \begin{displaymath}
       \mathcal{M}_f:=\big\{\ut^\alpha\cdot\ux^\beta\;|\;a_{\alpha,\beta}\not=0\big\}
     \end{displaymath}
     the set of \emph{monomials of $f$} and
     \begin{displaymath}
       \mathcal{T}_f:=\big\{a_{\alpha,\beta}\cdot \ut^\alpha\cdot\ux^\beta\;|\;a_{\alpha,\beta}\not=0\big\}
     \end{displaymath}
     the set of \emph{terms of $f$}.
     
     Moreover,
     \begin{displaymath}
       \lm_>(f):=\max\big\{\ut^\alpha\cdot\ux^\beta\;|\;\ut^\alpha\cdot\ux^\beta\in\mathcal{M}_f\big\}
     \end{displaymath}
     is called the \emph{leading monomial} of $f$. Note, that this maximum
     exists since  the number of $\beta$'s occurring in $f$ is finite and the
     ordering is local with respect to $\ut$.

     If $\lm_>(f)=\ut^\alpha\cdot \ux^\beta$ then we call
     \begin{displaymath}
       \lc_>(f):=a_{\alpha,\beta}
     \end{displaymath}
     the \emph{leading coefficient} of $f$,
     \begin{displaymath}
       \lt_>(f):=a_{\alpha,\beta}\cdot\ut^\alpha\cdot\ux^\beta
     \end{displaymath}
     its \emph{leading term}, and
     \begin{displaymath}
       \TAIL_>(f):=f-\lt_>(f)
     \end{displaymath}
     its \emph{tail}.

     For the sake of completeness we define
     \begin{displaymath}
       \lm_>(0):=0,\;\;\;\lt_>(0):=0,\;\;\;\lc_>(0):=0,\;\;\;\TAIL_>(f)=0,
     \end{displaymath}
     and
     \begin{displaymath}
       0<\ut^\alpha\cdot\ux^\beta\;\;\;\forall\;\alpha\in\N^m,\beta\in\N^n.
     \end{displaymath}

     Finally, for a subset $G\subseteq R[\ux]$ we call  the ideal 
     \begin{displaymath}
       L_>(G)=\langle \lm_>(f)\;|\;f\in G\rangle\lhd K[\ut,\ux]
     \end{displaymath}
     in the polynomial ring $K[\ut,\ux]$ generated by all the leading
     monomials of elements in $G$ the \emph{leading ideal} of $G$. 
   \end{definition}

   Analogously we define the notions for elements in $R[\ux]^s$.
   }

   \begin{definition}
     Let $>$ be a $\ut$-local monomial ordering on $\Mon^s(\ut,\ux)$.
     We call
     \begin{displaymath}
       0\not=f=\sum_{i=1}^s\sum_{|\beta|=0}^d \sum_{|\alpha|=0}^\infty
       a_{\alpha,\beta,i}\cdot \ut^\alpha\cdot \ux^\beta\cdot e_i\in R[\ux]^s,       
     \end{displaymath}
     with
     $a_{\alpha,\beta,i}\in K$, $|\beta|=\beta_1+\ldots+\beta_n$
     and $|\alpha|=\alpha_1+\ldots+\alpha_m$, the \emph{distributive
       representation} of $f$, 
     \bmath
       \mathcal{M}_f:=\big\{\ut^\alpha\cdot\ux^\beta\cdot e_i\;|\;a_{\alpha,\beta,i}\not=0\big\}
     \emath
     the set of \emph{monomials of $f$} and
     \bmath
       \mathcal{T}_f:=\big\{a_{\alpha,\beta,i}\cdot
       \ut^\alpha\cdot\ux^\beta\cdot e_i\;|\;a_{\alpha,\beta,i}\not=0\big\}
     \emath
     the set of \emph{terms of $f$}.
     
     Moreover,
     \bmath
       \lm_>(f):=\max\{\ut^\alpha\cdot\ux^\beta\cdot
       e_i\;|\;\ut^\alpha\cdot\ux^\beta\cdot e_i\in\mathcal{M}_f\} 
     \emath
     is called the \emph{leading monomial} of $f$. Note again, that this maximum
     exists since  the number of $\beta$'s occurring in $f$ and the
     number of $i$'s is finite and the
     ordering is local with respect to $\ut$.

     If $\lm_>(f)=\ut^\alpha\cdot \ux^\beta\cdot e_i$ then we call
     \bmath
       \lc_>(f):=a_{\alpha,\beta,i}
     \emath
     the \emph{leading coefficient} of $f$,
     \bmath
       \lt_>(f):=a_{\alpha,\beta,i}\cdot\ut^\alpha\cdot\ux^\beta\cdot e_i
     \emath
     its \emph{leading term}, and
     \bmath
       \TAIL_>(f):=f-\lt_>(f)
     \emath
     its \emph{tail}.

     For the sake of completeness we \lang{again} define
     \bmath
       \lm_>(0):=0,\lang{\;\;}\;\lt_>(0):=0,\lang{\;\;}\;\lc_>(0):=0,\lang{\;\;}\;\TAIL_>(f)=0,
     \emath
     and
     \bmath
       0<\ut^\alpha\cdot\ux^\beta\cdot e_i\;\;\;\forall\;\alpha\in\N^m,\beta\in\N^n,i\in\N.
     \emath

     Finally, for a subset $G\subseteq R[\ux]^s$ we call  the submodule
     \begin{displaymath}
       L_>(G)=\langle \lm_>(f)\;|\;f\in G\rangle\leq K[\ut,\ux]^s
     \end{displaymath}
     of the free module $K[\ut,\ux]^s$ over the polynomial ring
     $K[\ut,\ux]$ generated by all the leading 
     monomials of elements in $G$ the \emph{leading submodule} of $G$. 
   \end{definition}

   \lang{
   Since the monomial ordering is compatible with the semi group
   structure on $\Mon(\ut,\ux)$ respectively with the operation of
   the semi group $\Mon(\ut,\ux)$ on $\Mon^s(\ut,\ux)$ the statements in the following lemma
   are obvious.

   \begin{lemma}
     Let $f\in R[\ux]$ and $g\in R[\ux]^s$. 
     \begin{enumerate}
     \item $\lm_>(f\cdot g)=\lm_>(f)\cdot \lm_>(g)$,
     \item $\lc_>(f\cdot g)=\lc_>(f)\cdot \lc_>(g)$,
     \item $\lt_>(f\cdot g)=\lt_>(f)\cdot \lt_>(g)$.
     \end{enumerate}
   \end{lemma}
   \begin{proof}
     Since the statements are true when $f$ or $g$ is zero, we may
     assume that neither of them is so. Note that for any terms
     $a_{\alpha,\beta}\cdot\ut^\alpha\cdot\ux^\beta$ of $f$ and
     $b_{\alpha',\beta',i}\cdot\ut^{\alpha'}\cdot\ux^{\beta'}\cdot e_i$ of $g$ we
     have
     \begin{displaymath}
       \lm_>(f)\cdot\lm_>(g)\geq \ut^\alpha\cdot
       \ux^\beta\cdot\ut^{\alpha'}\cdot\ux^{\beta'}\cdot e_i
     \end{displaymath}
     with equality if and only if $\lm_>(f)=\ut^\alpha\cdot \ux^\beta$
     and $\lm_>(g)=\ut^{\alpha'}\cdot\ux^{\beta'}\cdot e_i$. This proves the lemma.
   \end{proof}
   }

   We know that in general a standard basis of an ideal respectively submodule $I$ will not be a
   generating set of $I$ itself, but only of the ideal respectively submodule which
   $I$ generates in the localisation with respect to the monomial
   ordering. We therefore introduce this notion here as well.

   \begin{definition}
     Let $>$ be a $\ut$-local monomial ordering on $\Mon(\ut,\ux)$,
     then 
     \bmath
       S_>=\{u\in R[\ux]\;|\;\lt_>(u)=1\}
     \emath
     is the \emph{multiplicative set associated to $>$}, and 
     \bmath
       R[\ux]_>=S_>^{-1}R[\ux]=\left\{\frac{f}{u}\;\Big|\;f\in R[\ux],u\in S_>\right\}
     \emath
     is the \emph{localisation of $R[\ux]$ with respect to $>$}.
   \end{definition}

   \lang{
   \begin{remark}
   }
     If $>$ is a $\ut$-local monomial ordering with 
     $x_i>1$ for all $i=1,\ldots,n$  (e.g.\
     $>_{lex}$ from Example \ref{ex:lex}), then
     $S_>\subset R^*$, and therefore $R[\ux]_>=R[\ux]$.
   \lang{
   \end{remark}
   }

   It is straight forward to extend the notions of leading monomial,
   leading term and leading coefficient to $R[\ux]_>$ and free modules
   over this ring.

   \begin{definition}
     Let $>$ be a $\ut$-local monomial ordering on $\Mon^s(\ut,\ux)$,
     $g=\frac{f}{u}\in R[\ux]_>^s$ with $u\in S_>$, and $G\subseteq
     R[\ux]_>^s$. We then define 
     the \emph{leading monomial}, the \emph{leading coefficient}
     respectively the \emph{leading term} of $g$ as
     \begin{displaymath}
       \lm_>(g):=\lm_>(f),\;\;\;
       \lc_>(g):=\lc_>(f),\;\;\;\mbox{ resp. }\;\;\;
       \lt_>(g):=\lt_>(f),       
     \end{displaymath}
     and the \emph{leading ideal} (if $s=1$) respectively
     \emph{leading submodule} of $G$
     \begin{displaymath}
       L_>(G)=\langle \lm_>(h)\;|\;h\in G\rangle\leq K[\ut,\ux]^s.
     \end{displaymath}
     These definitions are independent of the chosen representative,
     since if $g=\frac{f}{u}=\frac{f'}{u'}$ then $u'\cdot f=u\cdot f'$,
     and hence
     \begin{displaymath}
       \lt_>(f)=\lt_>(u')\cdot\lt_>(f)=\lt_>(u'\cdot f)=\lt_>(u\cdot f')
       =\lt_>(u)\cdot\lt_>(f')=\lt_>(f').
     \end{displaymath}
   \end{definition}

   \begin{remark}\label{rem:leadingsubmodule}
     Note that the leading submodule of a submodule in $R[\ux]_>^s$
     is a submodule in a free module over the polynomial ring $K[\ut,\ux]$ over the base
     field, and note that for $J\leq R[\ux]_>^s$ we obviously have
     \bmath
       L_>(J)=L_>(J\cap R[\ux]^s),
     \emath
     and similarly for $I\leq R[\ux]^s$ we have
     \bmath
       L_>(I)=L_>\big(\langle I\rangle_{R[\ux]_>}\big),
     \emath
     since every element of $\langle I\rangle_{R[\ux]_>}$ is of the
     form $\frac{f}{u}$ with $f\in I$ and $u\in S_>$.
     \lang{\hfill$\Box$}
   \end{remark}

   In order to be able to work either theoretically or even
   computationally with standard bases it is vital to have a division with
   remainder and possibly an algorithm to compute it. We will therefore
   generalise Grauert's and Mora's Division
   with remainder. For this we first would like to consider the
   different qualities a division with remainder may satisfy.

   \begin{definition}     
     Let $>$ be a $\ut$-local monomial ordering on $\Mon^s(\ut,\ux)$,
     and let $A=R[\ux]$ or $A=R[\ux]_>$, where we consider the latter
     as a subring of $K[[\ut,\ux]]$ in order to have the notion of
     terms of elements at hand.

     Suppose we have  $f,g_1,\ldots,g_k,r\in A^s$ and
     $q_1,\ldots,q_k\in A$ such that
     \begin{equation}\label{eq:standrep:1}
       f=q_1\cdot g_1+\ldots+q_k\cdot g_k+r.
     \end{equation}
     With the notation $r=\sum_{j=1}^sr_j\cdot e_j$,
     $r_1,\ldots,r_s\in A$, we say that \eqref{eq:standrep:1}
     satisfies with respect to $>$ 
     the condition
     \begin{enumerate}
     \item[(ID1)] iff $\lm_>(f)\geq\lm_>(q_i\cdot g_i)$ for all
       $i=1,\ldots,k$,
     \item[(ID2)] iff $\lm_>(g_i)\;\not|\;\lm_>(r)$ for $i=1,\ldots,k$,
       unless $r=0$,
     \item[(DD1)] iff for $j<i$ no term of $q_i\cdot\lm_>(g_i)$ is
       divisible by $\lm_>(g_j)$,
     \item[(DD2)] iff no term of $r$ is divisible by $\lm_>(g_i)$ for
       $i=1,\ldots,k$. 
     \item[(SID2)] iff $\lm_>(g_i)\;\not|\;\lm_>(r_j\cdot e_j)$ unless
       $r_j=0$ for all $i$ and $j$.
     \end{enumerate}
     Here, ``ID'' stands for \emph{indeterminate division with
       remainder} while ``DD'' means \emph{determinate division with
       remainder} and the ``S'' in (SID2) represents
     \emph{strong}. Accordingly, we call a representation of $f$ as in     
     \eqref{eq:standrep:1} a \emph{determinate division with
       remainder} 
     of $f$ with respect to $(g_1,\ldots,g_k)$ if 
     it satisfies (DD1) and (DD2), while we call it an
     \emph{indeterminate division with remainder} 
     of $f$ with respect to $(g_1,\ldots,g_k)$ if it satisfies (ID1) and
     (ID2). In any of these cases we call $r$ a \emph{remainder} or a
     \emph{normal form} of $f$ with respect to $(g_1,\ldots,g_k)$.

     If the remainder in a division with remainder of $f$ with respect
     to $(g_1,\ldots,g_k)$ is zero we call the representation of $f$ a
     \emph{standard representation}.

     Finally, if $A=R[\ux]$ then for $u\in S_>$ we call a division with remainder of
     $u\cdot f$ with respect to $(g_1,\ldots,g_k)$ also a \emph{weak
       division with remainder} of $f$ with respect to $(g_1,\ldots,g_k)$, a
     remainder of $u\cdot f$ with respect to $(g_1,\ldots,g_k)$ is
     called a \emph{weak normal form} of $f$ with respect to
     $(g_1,\ldots,g_k)$, and a standard representation of $u\cdot f$ with
     respect to $(g_1,\ldots,g_k)$ is called a \emph{weak standard
       representation} of $f$ with respect to $(g_1,\ldots,g_k)$.
   \end{definition}

   \kurz{
     It is rather obvious to see that (DD2) $\Longleftarrow$
     (SID2)$\Longleftarrow$ (ID2), that (DD1)+(ID2) $\Longleftarrow$
     (ID1), and that the coefficients and the remainder of a division
     satisfying (DD1) and (DD2) is uniquely determined. 
   }

   \lang{
   The following lemma should clarify the relations between the above
   conditions, and it should explain the \emph{determinate} versa the
   \emph{indeterminate}. 

   \begin{lemma}\label{lem:conditions}
     Let $>$ be a $\ut$-local monomial ordering on $\Mon^s(\ut,\ux)$,
     and suppose we have $f,g_1, \ldots,g_k, r\in R[\ux]^s$ and 
     $q_1,\ldots,q_k\in R[\ux]$ such that
     \begin{equation}\label{eq:standrep:2}
       f=q_1\cdot g_1+\ldots+q_k\cdot g_k+r.
     \end{equation}
     Then the following holds true:
     \begin{enumerate}
     \item If \eqref{eq:standrep:2} satisfies (DD2) then it satisfies
       (SID2). 
     \item If \eqref{eq:standrep:2} satisfies (SID2) then it satisfies
       (ID2).
     \item If \eqref{eq:standrep:2} satisfies (DD1) and (ID2) then it
       satisfies (ID1).
     \item Suppose that $g_i\not=0$ for $i=1,\ldots,k$. 
       If $f=q_1'\cdot g_1+\ldots+q_k'\cdot q_k+r'$ is a second
       such representation of $f$ and both satisfy (DD1) and (DD2),
       then $q_1=q_1',\ldots,q_k=q_k'$ and $r=r'$. That is, a
       determinate division with remainder is uniquely determined, if
       it exists.
     \end{enumerate}
   \end{lemma}
   \begin{proof}
     \begin{enumerate}
     \item This is obvious, if we take into account that $0$ has no
       term and therefore, even if $r=0$ it is true that no term of
       $r$ is divisible by any $\lm_>(g_i)$.
     \item This is obvious.
     \item Let $\lm_>(q_j\cdot g_j)=\ut^\alpha\cdot\ux^\beta\cdot
       e_\nu$ be maximal in $\{\lm_>(q_i\cdot 
       g_i)\;|\; i=1,\ldots,k\}$, and suppose that $\lm_>(q_j\cdot
       g_j)>\lm_>(f)$.  Since the
       monomial $\ut^\alpha\cdot\ux^\beta\cdot e_\nu$ does not occur on the left
       hand side of the equality in \eqref{eq:standrep:2} it must
       occur at least in one of the other summands on the right hand
       side in order to cancel. Suppose it occurs in some $q_i\cdot
       g_i$, $i\not=j$, then by our choice necessarily
       \begin{displaymath}
         \lm_>(q_j)\cdot\lm_>(g_j)=\lm_>(q_j\cdot g_j)=\lm_>(q_i\cdot g_i)=\lm_>(q_i)\cdot\lm_>(g_i),
       \end{displaymath}
       in contradiction to (DD1). Thus $\ut^\alpha\cdot\ux^\beta\cdot e_\nu$ must
       be a term in $r$, and again by the maximality in our choice
       necessarily
       \begin{displaymath}
         \lm_>(q_j)\cdot\lm_>(g_j)=\lm_>(q_j\cdot g_j)=\lm_>(r),
       \end{displaymath}
       in contradiction to (ID2). This shows that (ID1) is satisfied.
     \item Obviously the representation
       \begin{displaymath}
         0=f-f=(q_1-q_1')\cdot g_1+\ldots+(q_k-q_k')\cdot g_k+(r-r')
       \end{displaymath}
       still satisfies the conditions (DD1) and (DD2). But then by (a)
       and (b) it satisfies (ID2) and by (c) it satisfies (ID1). This implies
       \begin{displaymath}
         0=\lm_>(0)\geq \lm_>(q_i-q_i')\cdot \lm_>(g_i)\geq 0,
       \end{displaymath}
       and since by our assumption $g_i\not=0$ this implies
       $\lm_>(q_i-q_i')=0$ and hence $q_i=q_i'$. And therefore,
       finally, also $r=r'$.
     \end{enumerate}
   \end{proof}
   }

   We first want to generalise Grauert's Division with Remainder to
   the case of elements in $R[\ux]$ which are homogeneous with
   respect to $\ux$. We therefore introduce this notion in the
   following definition.

   \begin{definition}
     Let
     $f=\sum_{i=1}^s\sum_{|\beta|=0}^d\sum_{\alpha\in\N^m}
     a_{\alpha,\beta,i}\cdot\ut^\alpha\cdot\ux^\beta\cdot e_i
     \in R[\ux]^s$. 
     \begin{enumerate}
     \item We call
       \bmath
         \deg_{\ux}(f):=\max\big\{|\beta|\;\big|\;a_{\alpha,\beta,i}\not=0\big\}
       \emath
       the \emph{$\ux$-degree} of $f$.
     \item $f\in R[\ux]^s$ is called \emph{$\ux$-homogeneous of
         $\ux$-degree $d$} 
       \lang{
         if for any $0\not=\lambda\in K$ we have
       \begin{displaymath}
         f(\ut,\lambda\cdot\ux)=\lambda^d\cdot f,
       \end{displaymath}
       or, equivalently,
       }
       if all terms of $f$
       have the same
       $\ux$-degree $d$. We denote by $R[\ux]_d^s$ the
       $R$-submodule of $R[\ux]^s$ of $\ux$-homogeneous elements.
       Note that by this definition $0$ is $\ux$-homogeneous of degree
       $d$ for all $d\in\N$.
     \item If $>$ is a $\ut$-local monomial ordering on $\Mon^s(\ut,\ux)$
       then we call 
       \begin{displaymath}
         \ecart_>(f):=\deg_{\ux}(f)-\deg_{\ux}\big(\lm_>(f)\big)\geq 0
       \end{displaymath}
       the \emph{ecart} of $f$. It in some sense measures the
       failure of the homogeneity of $f$.
     \end{enumerate}
   \end{definition}
   

   \section{Determinate Division with Remainder in $K[[\ut]][\ux]_d^s$}\label{sec:HDDwR}

   We are now ready to show that for $\ux$-homogeneous elements in
   $R[\ux]$ there exists a determinate division with remainder. We
   follow mainly the proof of Grauert's Division Theorem as given in
   \cite{DS07}. 

   \begin{theorem}[HDDwR]\label{thm:HDDwR}
     Let $f,g_1,\ldots,g_k\in R[\ux]^s$ be $\ux$-homogeneous, then there
     exist uniquely determined $q_1,\ldots,q_k\in R[\ux]$ and $r\in R[\ux]^s$ such that
     \begin{displaymath}
       f=q_1\cdot g_1+\ldots +q_k\cdot g_k+r
     \end{displaymath}
     satisfying
     \kurz{(DD1), (DD2) and}
     \begin{enumerate}
     \lang{
     \item[(DD1)] For $j<i$ no term of $q_i\cdot\lm_>(g_i)$ is divisible
       by $\lm_>(g_j)$,
     \item[(DD2)]
       no term of $r$ is divisible by any $\lm_>(g_i)$, and 
     }
     \item[(DDH)] $q_1,\ldots,q_k,r$ are $\ux$-homogeneous of
       $\ux$-degrees $\deg_{\ux}(q_i)=\deg_{\ux}(f)-\deg_{\ux}\big(\lm_>(g_i)\big)$
       respectively $\deg_{\ux}(r)=\deg_{\ux}(f)$.
     \end{enumerate}
     \lang{
     We call such a representation of $f$ a \emph{homogeneous
       determinate division with remainder}. 
     }
   \end{theorem}
   \begin{proof}
     \kurz{ The result is obvious if the $g_i$ are terms, and we will
       reduce the general case to this one.}
     \lang{
     We first consider the case that $g_1,\ldots,g_k\in\Mon^s(\ut,\ux)$
     are monomials. Then  define recursively
     for $i=1,\ldots,k$
     \begin{displaymath}
       q_i:=\frac{h_i}{g_i}\in R[\ux],
     \end{displaymath}
     where $h_i$ is the sum of all terms of
     $f-\sum_{j=1}^{i-1}q_j\cdot g_j$ which
     are divisible by $g_i$, $i=1,\ldots,k$. Thus with $r:=f-\sum_{i=1}^k q_i\cdot g_i$
     \begin{displaymath}
       f=q_1\cdot g_1+\ldots +q_k\cdot g_k+r
     \end{displaymath}
     obviously satisfies the above conditions (DD1), (DD2) and (DDH).
     
     This result generalises immediately to the case where the $g_i$
     are terms, i.e.\ constant multiples of monomials.     

     Let us now consider the general case.} We set $f_0=f$ and for
     $\nu>0$ we define recursively 
     \begin{displaymath}
       f_\nu=f_{\nu-1}-\sum_{i=1}^k q_{i,\nu}\cdot
       g_i-r_\nu=\sum_{i=1}^k q_{i,\nu}\cdot(\big(-\TAIL(g_i)\big),
     \end{displaymath}
     where the $q_{i,\nu}\in R[\ux]$ and $r_\nu\in R[\ux]^s$ are such that
     \begin{equation}\label{eq:HDDwR:1}
       f_{\nu-1}=q_{1,\nu}\cdot\lt_>(g_1)+\ldots+q_{k,\nu}\cdot\lt_>(g_k)+r_\nu
     \end{equation}
     satisfies (DD1), (DD2) and (DDH). Note that such a representation
     of $f_{\nu-1}$ exists since the $\lt_>(g_i)$ are terms.

     We want to show that $f_\nu$, $q_{i,\nu}$ and $r_\nu$ all converge to zero in the
     $\langle t_1,\ldots,t_m\rangle$-adic topology, that is that for each $N\geq 0$
     there exists a $\mu_N\geq 0$ such that for all $\nu\geq \mu_N$
     \begin{displaymath}
       f_\nu,r_\nu\in\langle t_1,\ldots,t_m\rangle^N\cdot R[\ux]^s 
       \;\;\;\mbox{ resp. }\;\;\;
       q_{i,\nu}\in\langle t_1,\ldots,t_m\rangle^N.
     \end{displaymath}
     
     By Lemma \ref{lem:weightedorder} there is $\ut$-local weight
     ordering $>_w$ such that 
     \begin{displaymath}
       \lm_>(g_i)=\lm_{>_w}(g_i)\;\;\;\mbox{ for all }i=1,\ldots,k. 
     \end{displaymath}
     If we replace in the above construction $>$ by $>_w$,
     we still get the same sequences $(f_\nu)_{\nu=0}^\infty$,
     $(q_{i,\nu})_{\nu=1}^\infty$ and $(r_\nu)_{\nu=1}^\infty$, since for the
     construction of $q_{i,\nu}$ and $r_\nu$ only the leading
     monomials of the $g_j$ are used. In particular,
     \eqref{eq:HDDwR:1} will satisfy (DD1), (DD2) and (DDH) with
     respect to $>_w$.
     Due to (DDH) $f_\nu$ is again $\ux$-homogeneous of
     $\ux$-degree equal to that of $f_{\nu-1}$, and
     since (DD1) and (DD2) imply (ID1) \lang{by Lemma \ref{lem:conditions} }
     we have
     \begin{multline*}
       \lm_{>_w}(f_{\nu-1})\geq
       \max\{\lm_{>_w}(q_{i,\nu})\cdot\lm_{>_w}(g_i)\;|\;i=1,\ldots,k\}\\
       >\max\big\{\lm_{>_w}(q_{i,\nu})\cdot\lm_{>_w}\big(-\TAIL(g_i)\big)\;\big|\; i=1,\ldots,k\big\}
       \geq\lm_{>_w}(f_\nu).
     \end{multline*}
     It follows from Lemma \ref{lem:convergence} that $f_\nu$
     converges to zero in the $\langle t_1,\ldots,t_m\rangle$-adic
     topology, i.e.\ for given $N$ there is a $\mu_N$ such that
     \begin{displaymath}
       f_\nu\in\langle t_1,\ldots,t_m\rangle^N\cdot R[\ux]^s \;\;\;\mbox{ for all }\nu\geq\mu.
     \end{displaymath}
     But then, by construction for $\nu>\mu_N$
     \begin{displaymath}
       r_\nu\in\langle t_1,\ldots,t_m\rangle^N\cdot R[\ux]^s
     \end{displaymath}
     and
     \begin{displaymath}
       q_{i,\nu}\in\langle t_1,\ldots,t_m\rangle^{N-d_i},
     \end{displaymath}
     where
     $d_i=\deg\big(\lm_>(g_i)\big)-\deg_{\ux}\big(\lm_>(g_i)\big)$ is
     independent of $\nu$.
     Thus both, $r_\nu$ and $q_{i,\nu}$, converge as well to zero in
     the $\langle t_1,\ldots,t_m\rangle$-adic 
     topology.

     But then 
     \begin{displaymath}
       q_i:=\sum_{\nu=1}^\infty q_{i,\nu}\in R[\ux] \;\;\;\mbox{ and
       }\;\;\;
       r:=\sum_{\nu=1}^\infty r_\nu\in R[\ux]^s
     \end{displaymath}
     are $\ux$-homogeneous of $\ux$-degrees
     $\deg_{\ux}(q_i)=\deg_{\ux}(f)-\deg_{\ux}\big(\lm_>(g_i)\big)$ respectively
     $\deg_{\ux}(r)=\deg_{\ux}(f)$ unless they are zero, and
     \begin{displaymath}
       f=q_1\cdot g_1+\ldots+q_k\cdot g_k+r
     \end{displaymath}
     satisfies (DD1), (DD2) and (DDH).

     The uniqueness of the representation \kurz{is obvious.}
     \lang{follows from Lemma \ref{lem:conditions}.}
   \end{proof}

   The following lemmata contain technical results used throughout the proof
   of the previous theorem.

   \lang{
   \begin{lemma}\label{lem:graph}
     Let $\Gamma$ be a directed graph with vertex set $V$ and edge set
     $E\subset V\times V$. If for every vertex the number of
     outward pointing edges coincides with the number of inward
     pointing edges, then $F$ is a disjoint union of cycles.
   \end{lemma}
   \begin{proof}
     We do the proof by induction on the number $\# E$ of edges. If
     there is no edge, then the statement holds by default, and we may
     thus assume that $\# E>0$. Choose any vertex, say $v_0\in V$,
     which has an outward pointing edge, say $(v_0,v_1)\in E$. By the
     assumption $v_1$ also has an outward pointing edge, say
     $(v_1,v_2)$, and we can inductively proceed to construct a
     sequence of vertices $(v_\nu)_{\nu\in\N}$ with
     $(v_\nu,v_{\nu+1})\in E$. Since the number of vertices is finite,
     there is minimal $\mu\geq 0$ such that 
     \begin{displaymath}
       v_\mu=v_\nu \;\;\;\mbox{ for some }\;\;\;0\leq \nu<\mu.
     \end{displaymath}
     But then 
     \begin{displaymath}
       C=\big((v_\nu,v_{\nu+1}),\ldots,(v_{\mu-1},v_\mu)\big)
     \end{displaymath}
     is a cycle, and if we remove $C$ from $E$ then the remaining
     graph still satisfies that for each vertex the number of inward
     pointing and outward pointing edges coincides. Thus by induction
     it is a disjoint union of cycles and then so is $\Gamma$.
   \end{proof}

   \begin{lemma}\label{lem:bayer}
     We use the notation $\uz=(\ut,\ux)$ and
     $\E^s=\{e_1,\ldots,e_s\}$.

     Let $>$ be a monomial ordering on $\Mon^s(\uz)$ and consider the
     set
     \begin{displaymath}
       \Delta_>:=\big\{(\gamma,e_i)-(\gamma',e_j)\in\Z^{m+n+s}\;|\;
       \uz^\gamma\cdot e_i>\uz^{\gamma'}\cdot e_j\big\}.
     \end{displaymath}
     Then:
     \begin{displaymath}
       (0,\ldots,0)\not\in
       \left\{\sum_{i=1}^kn_i\cdot\delta_i\;\bigg|\;\delta_i\in \Delta_>,n_i\in\Z_{>0},k>0\right\}.
     \end{displaymath}
   \end{lemma}
   \begin{proof}
     For the convenience of the reader we reproduce here the proof
     given in Dave Bayer's thesis \cite[(1.7)]{Bay82}.
     Suppose there exist, not necessarily pairwise different, 
     \begin{displaymath}
       \delta_i=(\gamma_{i,2},e_{j_{i,2}})-(\gamma_{i,1},e_{j_{i,1}})\in
       \Delta_>, \;\;\;i=1,\ldots,k,
     \end{displaymath}
     with $\gamma_{i,1},\gamma_{i,2}\in\N^{m+n}$,
     $j_{i,1},j_{i,2}\in\{1,\ldots,s\}$ and
     \begin{displaymath}
       \uz^{\gamma_{i,2}}\cdot e_{j_{i,2}}\;>\;\uz^{\gamma_{i,1}}\cdot e_{j_{i,1}},
     \end{displaymath}
     such that
     \begin{displaymath}
       \left(\sum_{i=1}^k\gamma_{i,2}-\gamma_{i,1},\sum_{i=1}^ke_{j_{i,2}}-e_{j_{i,1}}\right)
       =\sum_{i=1}^k\delta_i=(0,\ldots,0).  
     \end{displaymath}
     
     It is our first aim to show that we may assume that
     $e_{j_{i,1}}=e_{j_{i,2}}$ for all $i=1,\ldots,k$.

     For this we define a directed graph $\Gamma$ whose vertex set is
     $\E^s$ and such that for each $i=1,\ldots,k$ there is an edge from $e_{j_{i,1}}$
     to $e_{j_{i,2}}$. Since by assumption
     $\sum_{i=1}^k e_{j_{i,2}}-e_{j_{i,1}}=(0,\ldots,0)$,  for each vertex $e_\nu$
     the number of edges pointing towards $e_\nu$ is equal to the number
     of edges pointing away from $e_i$. Thus by Lemma \ref{lem:graph} $\Gamma$
     is a disjoint union of cycles, each given by a subset of
     $\{\delta_1,\ldots,\delta_k\}$. Let
     $\{\delta_{i_1},\ldots,\delta_{i_\mu}\}$ represent such a cycle
     with 
     \begin{equation}\label{eq:bayer:1}
       j_{i_1,2}=j_{i_2,1},\ldots,j_{i_{\mu-1},2}=j_{i_\mu,1},j_{i_\mu,2}=j_{i_1,1}.
     \end{equation}
     Set $\gamma=\gamma_{i_1,1}+\ldots+\gamma_{i_\mu,1}\in\N^{m+n}$,
     $\varepsilon_0:=(\gamma,e_{j_{i_1,1}})$ and recursively 
     \begin{displaymath}
       \varepsilon_\nu:=\varepsilon_{\nu-1}+\delta_{i_\nu}=
       \left(\sum_{\kappa=1}^\nu(\gamma_{\kappa,2}-\gamma_{\kappa,1})+\gamma,e_{j_{i_\nu,2}}\right)
       \in\N^{m+n}\times \E^s
     \end{displaymath}
     for $\nu=1,\ldots,\mu$. By assumption
     \begin{displaymath}
       \uz^{\gamma_{i_\nu,2}}\cdot e_{j_{i_\nu,2}}\;>\;
       \uz^{\gamma_{i_\nu,1}}\cdot e_{j_{i_\nu,1}}
     \end{displaymath}
     and multiplying both sides with 
     $\uz^{\sum_{\kappa=1}^{\nu-1}(\gamma_{i_\kappa,2}-\gamma_{i_\kappa,1})
       +\gamma-\gamma_{i_\nu,1}}\in\Mon(\uz)$ we get
     \begin{displaymath}
       \uz^{\sum_{\kappa=1}^\nu(\gamma_{i_\kappa,2}-\gamma_{i_\kappa,1})+\gamma}\cdot e_{j_{i_\nu,2}}
       \;>\;
       \uz^{\sum_{\kappa=1}^{\nu-1}(\gamma_{i_\kappa,2}-\gamma_{i_\kappa,1})+\gamma}\cdot e_{j_{i_\nu,1}}
     \end{displaymath}
     for $\nu=1,\ldots,\mu$. Transitivity of the monomial ordering and
     \eqref{eq:bayer:1} imply then
     \begin{displaymath}
       \uz^{\sum_{\kappa=1}^\mu \gamma_{i_\kappa,2}}\cdot e_{j_{i_\mu,2}}=
       \uz^{\sum_{\kappa=1}^\mu(\gamma_{i_\kappa,2}-\gamma_{i_\kappa,1})+\gamma}\cdot e_{j_{i_\mu,2}}
       \;>\;
       \uz^{\gamma}\cdot e_{j_{i_1,1}}=\uz^\gamma\cdot e_{j_{i_\mu,2}}.
     \end{displaymath}
     But then 
     \begin{displaymath}
       \sum_{\nu=1}^\mu \delta_{i_\nu}=\varepsilon_\mu-\varepsilon_0\in\Delta_>,
     \end{displaymath}
     and we may replace $\{\delta_{i_1},\ldots,\delta_{i_\mu}\}$ by
     its sum which is an element in $\Delta_>$ such that the last $s$
     components are all zero. Doing this with each of the cycles whose
     disjoint union is $\Gamma$ we may assume that from the beginning
     $e_{j_{i,1}}=e_{j_{i,2}}$ for all $i=1,\ldots,k$.

     If $e_{j_{i,1}}=e_{j_{i,2}}$ then $\uz^{\gamma_{i,2}}\cdot e_{j_{i,2}}>
     \uz^{\gamma_{i,1}}\cdot e_{j_{i,1}}$ implies
     \begin{displaymath}
       \uz^{\gamma_{i,2}}\;>\;\uz^{\gamma_{i,1}}
     \end{displaymath}
     with respect to the monomial ordering which $>$ induces on
     $\Mon(\uz)$. The compatibility of $>$ with respect to the
     semi group structure of $\Mon(\uz)$ then leads to
     \begin{displaymath}
       \uz^{\sum_{i=1}^k\gamma_{i,2}}\;>\;\uz^{\sum_{i=1}^k\gamma_{i,1}},
     \end{displaymath}
     while $\sum_{i=1}^k\delta_i=0$ implies
     \begin{displaymath}
       \sum_{i=1}^k\gamma_{i,2}=\sum_{i=1}^k\gamma_{i,1},
     \end{displaymath}
     which gives the desired contradiction.
   \end{proof}
   }

   \begin{lemma}\label{lem:weightvector}
     If $>$ is a monomial ordering on $\Mon^s(\uz)$ with
     $\uz=(\ut,\ux)$, and  $M\subset \Mon^s(\uz)$ is finite, then
     there exists $w\in\Z^{m+n+s}$ with
     \begin{displaymath}
       w_i<0, \;\;\mbox{ if }z_i<1,\;\;\;\mbox{ and }\;\;\;
       w_i>0, \;\;\mbox{ if }z_i>1,
     \end{displaymath}
     such that for $\uz^\gamma\cdot e_i,
     \uz^{\gamma'}\cdot e_j\in M$ we have 
     \begin{displaymath}
       \uz^\gamma\cdot e_i\;>\;
       \uz^{\gamma'}\cdot e_j
       \;\;\;\Longleftrightarrow\;\;\;
       w\cdot(\gamma,e_i)\;>\;w\cdot(\gamma',e_j).
     \end{displaymath}
     In particular, if $>$ is $\ut$-local then every $\ut$-local weight ordering on
     $\Mon^s(\ut,\ux)$ with weight vector $w$ coincides on $M$ with $>$.          
   \end{lemma}
   \begin{proof}
     \kurz{
       The proof goes analogous to \cite[Lemma 1.2.11]{GP02}, using
       \cite[(1.7)]{Bay82} (for this note that in the latter the
       requirement that $>$ is a well-ordering is superfluous).
     }
     \lang{
     We set $M'=M\cup\{e_1,z_1\cdot e_1,\ldots,z_{m+n}\cdot
     e_1\}$, and we consider the finite subset
     \begin{displaymath}
       \Delta_M:=\big\{(\gamma,e_i)-(\gamma',e_j)\;|\;\uz^\gamma\cdot
       e_i, \uz^{\gamma'}\cdot e_j\in M', \uz^\gamma\cdot e_i>
       \uz^{\gamma'}\cdot e_j\big\}\subset\Delta_>
     \end{displaymath}
     of the set $\Delta_>$ from Lemma \ref{lem:bayer}.

     Lemma \ref{lem:bayer} implies that the convex hull of $\Delta_M$ over $\Q$,
     \begin{displaymath}
       \conv_\Q(\Delta_M):=
       \left\{\sum_{\delta\in\Delta_M}\lambda_\delta\cdot\delta
         \;\bigg|\;\lambda_\delta\in\Q_{\geq 0},\sum_{\delta\in\Delta_M}\lambda_\delta=1
       \right\},
     \end{displaymath}
     does not contain the zero vector, since multiplying such a vanishing convex
     combination with the greatest common denominator of the
     coefficients would lead to a positive integer combination as
     excluded in Lemma \ref{lem:bayer}.

     However, a convex set like $\conv_\Q(\Delta_M)$ which does not
     contain zero lies in the positive half 
     space defined by a linear form
     \begin{equation}\label{eq:weightvector:1}
       l_w:\Q^{m+n+s}\longrightarrow\Q:v\mapsto w\cdot
       v=\sum_{i=1}^{m+n+s}w_i\cdot v_i
     \end{equation}
     given by $w\in\Q^{m+n+s}$ (see e.g.\ \cite[Thm\ 2.10]{Val76}), i.e.
     \begin{displaymath}
       \delta\in\Delta_M\;\;\;\Longleftrightarrow\;\;\;w\cdot\delta>0.
     \end{displaymath}
     Multiplying with the common denominator of its entries we may assume that
     $w\in\Z^{m+n+s}$.

     Moreover, if for $i=1,\ldots,m+n$ we set $\delta_i:=z_i\cdot
     e_1-e_1$, then by \eqref{eq:weightvector:1}
     \begin{displaymath}
       z_i<1 
       \;\;\;\Longleftrightarrow\;\;\;
       -\delta_i\in\Delta_M
       \;\;\;\Longleftrightarrow\;\;\;
       w_i=\delta_i.w_i<0
     \end{displaymath}
     and 
     \begin{displaymath}
       z_i>1 
       \;\;\;\Longleftrightarrow\;\;\;
       \delta_i\in\Delta_M
       \;\;\;\Longleftrightarrow\;\;\;
       w_i=\delta_i.w_i>0.
     \end{displaymath}     
     }
   \end{proof}

   \begin{lemma}\label{lem:weightedorder}
     Let $>$ be a $\ut$-local ordering on $\Mon^s(\ut,\ux)$ and
     let $g_1,\ldots,g_k\in R[\ux]^s$ be $\ux$-homogeneous (not
     necessarily of the same degree), then there is a $w\in
     \Z_{<0}^m\times\Z^{n+s}$ such that any $\ut$-local weight
     ordering with weight vector $w$, say $>_w$, induces the same
     leading monomials as $>$ on $g_1,\ldots,g_k$, i.e.
     \begin{displaymath}
       \lm_>(g_i)=\lm_{>_w}(g_i)\;\;\;\mbox{ for all }i=1,\ldots,k.
     \end{displaymath}
   \end{lemma}
   \begin{proof}
     Consider the monomial ideals $I_i=\langle \mathcal{M}_{\TAIL(g_i)}\rangle$ in
     $K[\ut,\ux]$ generated by all monomials of $\TAIL(g_i)$,
     $i=1,\ldots,k$. By Dickson's Lemma (see e.g.\ \cite[Lemma 1.2.6]{GP02}) $I_i$ is generated by
     a finite subset, say $B_i\subset \mathcal{M}_{\TAIL(g_i)}$, of
     the monomials of 
     $\TAIL(g_i)$. If we now set 
     \begin{displaymath}
       M=B_1\cup\ldots\cup B_k\cup\{\lm_>(g_1),\ldots,\lm_>(g_k)\},
     \end{displaymath}
     then by Lemma \ref{lem:weightvector} there is
     $w\in\Z_{<0}^m\times\Z^{n+s}$ such that any $\ut$-local weight
     ordering, say $>_w$, with weight vector $w$ coincides on $M$ with
     $>$. Let now $\ut^\alpha\cdot\ux^\beta\cdot e_\nu$ 
     be any monomial occurring in $\TAIL(g_i)$. Then there is a
     monomial $\ut^{\alpha'}\cdot\ux^{\beta'}\cdot e_\mu\in B_i$ such that 
     \begin{displaymath}
       \ut^{\alpha'}\cdot\ux^{\beta'}\cdot
       e_\mu\;\big|\;\ut^\alpha\cdot\ux^\beta\cdot e_\nu, 
     \end{displaymath}
     which in particular implies that $e_\nu=e_\mu$.
     Since $g_i$ is $\ux$-homogeneous it follows first that
     $|\beta|=|\beta'|$ and thus that $\beta=\beta'$. Moreover,
     since  $>_w$ is $\ut$-local it
     follows that $\ut^{\alpha'}\geq_w\ut^\alpha$ and thus that
     \begin{displaymath}
       \ut^{\alpha'}\cdot\ux^{\beta'}\cdot
       e_\mu\geq_w\ut^\alpha\cdot\ux^\beta\cdot e_\nu.
     \end{displaymath}
     But since $>$ and $>_w$ coincide on $\{\lm_>(g_i)\}\cup
     B_i\subset M$ we necessarily have that
     \begin{displaymath}
       \lm_>(g_i)\;>_w\;\ut^{\alpha'}\cdot\ux^{\beta'}\cdot e_\mu
       \geq_w\ut^\alpha\cdot\ux^\beta\cdot e_\nu,
     \end{displaymath}
     and hence $\lm_{>_w}(g_i)=\lm_>(g_i)$.
   \end{proof}

   \begin{lemma}\label{lem:convergence}
     Let $>$ be a $\ut$-local weight ordering on
     $\Mon^s(\ut,\ux)$ with weight vector $w\in\Z_{<0}^m\times\Z^{n+s}$, and let
     $(f_\nu)_{\nu\in\N}$ be a sequence of $\ux$-homogeneous elements
     of fixed $\ux$-degree $d$ in $R[\ux]^s$ such that
     \begin{displaymath}
       \lm_>(f_\nu)>\lm_>(f_{\nu+1})\;\;\;\mbox{ for all }\;\nu\in \N.
     \end{displaymath}
     Then $f_\nu$
     converges to zero in the $\langle t_1,\ldots,t_m\rangle$-adic topology, i.e.\ 
     \begin{displaymath}
       \forall\;N\geq 0\;\exists\;\mu_N\geq 0\;:\;\forall\;\nu\geq \mu_N\;
       \mbox{ we have }\;f_\nu\in\langle t_1,\ldots,t_m\rangle^N\cdot R[\ux]^s.
     \end{displaymath}
     In particular, the element $\sum_{\nu=0}^\infty f_\nu\in R[\ux]_d^s$
     exists.
   \end{lemma}
   \begin{proof}
     Since $w_1,\ldots,w_m<0$ the set of monomials
     \begin{displaymath}
       M_k=\big\{\ut^\alpha\cdot \ux^\beta\;\big|\;w\cdot(\alpha,\beta,e_i)> -k,|\beta|=d\big\}.
     \end{displaymath}
     is finite for a any fixed $k\in\N$. 

     Let $N\geq 0$ be fixed, set
     $\tau=\max\{|w_1|,\ldots,|w_{m+n+s}|\}$ and $k:=(N+nd+1)\cdot \tau$, then for
     any monomial $\ut^\alpha\cdot \ux^\beta\cdot e_j$ of $\ux$-degree $d$ 
     \begin{equation}\label{eq:convergence:1}
       \ut^\alpha\cdot \ux^\beta\cdot e_j\not\in M_k\;\;\;\Longrightarrow
       \;\;\;
       \ut^\alpha\cdot \ux^\beta\cdot e_j\in \langle
       t_1,\ldots,t_m\rangle^N\cdot R[\ux]^s,
     \end{equation}
     since
     \begin{displaymath}
       \sum_{i=1}^m\alpha_i\cdot w_i\leq
       -k-\sum_{i=1}^n\beta_i\cdot w_{m+i}
       -w_{m+n+j}
       \leq -k+(nd+1)\cdot\tau
     \end{displaymath}
     and thus
     \begin{displaymath}
       |\alpha|=\sum_{i=1}^m\alpha_i\geq
       \sum_{i=1}^m\alpha_i\cdot\frac{-w_i}{\tau}\geq
       \frac{k}{\tau}-nd-1=N.
     \end{displaymath}
     Moreover, since $M_k$ is finite and the $\lm_>(f_\nu)$ are
     pairwise different there are only finitely many
     $\nu$ such that $\lm_>(f_\nu)\in M_k$. Let $\mu$ be maximal among those
     $\nu$, then by \eqref{eq:convergence:1}
     \begin{displaymath}
       \lm_>(f_\nu)\in\langle t_1,\ldots,t_m\rangle^N\cdot R[\ux]^s\;\;\mbox{ for all }\;\; \nu>\mu.
     \end{displaymath}
     But since $>$ is a $\ut$-local weight ordering we have that
     $\lm_>(f_\nu)\not\in M_k$ implies that no monomial of $f_\nu$ is in
     $M_k$, and thus $f_\nu\in \langle t_1,\ldots,t_m\rangle^N\cdot
     R[\ux]^s$ for all $\nu>\mu$ 
     by \eqref{eq:convergence:1}. This shows that $f_\nu$ converges to
     zero in the $\langle t_1,\ldots,t_m\rangle$-adic topology.

     Since $f_\nu$ converges to zero in the $\langle
     t_1,\ldots,t_m\rangle$-adic topology, for every monomial
     $\ut^\alpha\cdot\ux^\beta\cdot e_j$ there is only a finite number of
     $\nu$'s such that $\ut^\alpha\cdot\ux^\beta\cdot e_j$ is a monomial
     occurring in $f_\nu$. Thus the sum $\sum_{\nu=0}^\infty f_\nu$
     exists and is obviously $\ux$-homogeneous of degree $d$.
   \end{proof}

   From the proof of Theorem \ref{thm:HDDwR} we can deduce an algorithm for
   computing the determinate division with remainder up to arbitrary
   order, or if we don't require termination then it will ``compute'' the
   determinate division with remainder completely. Since for our
   purposes termination is not important, we will simply formulate the
   non-terminating algorithm.

   \begin{algorithm}[$\HDDwR$]\label{alg:HDDwR}
     \textsc{Input:} 
     \begin{minipage}[t]{11cm}
       $(f,G)$ with 
       $G=\{g_1,\ldots,g_k\}$ and $f,g_1,\ldots,g_k\in R[\ux]^s$
       $\ux$-homogeneous, $>$ a $\ut$-local monomial ordering
     \end{minipage}
     \\[0.2cm]
     \textsc{Output:} 
     \begin{minipage}[t]{11cm}
       $(q_1,\ldots,q_k,r)\in R[\ux]^k\times R[\ux]^s$ such
       that
       \begin{displaymath}
         f=q_1\cdot g_1+\ldots +q_k\cdot g_k+r
       \end{displaymath}
       is a homogeneous determinate division with remainder of $f$ satisfying
       (DD1), (DD2) and (DDH).
     \end{minipage}
     \\[0.2cm]
     \textsc{Instructions:}
     \begin{itemize}
     \item $f_0:=f$
     \item $r:=0$       
     \item FOR $i=1,\ldots,k$ DO $q_i:=0$ 
     \item $\nu:=0$
     \item WHILE $f_\nu\not=0$ DO
       \begin{itemize}
       \item $q_{0,\nu}:=0$
       \item FOR $i=1,\ldots,k$ DO
         \begin{itemize}
         \item $h_{i,\nu}:=\sum_{p\in \mathcal{T}_{f_\nu}\;:\;\lm_>(g_i)\;|\;p}p$
         \item $q_{i,\nu}:=\frac{h_{i,\nu}}{\lt_>(g_i)}$
         \item $q_i:=q_i+q_{i,\nu}$
         \end{itemize}
       \item $r_\nu:=f_\nu-q_{1,\nu}\cdot\lt_>(g_1)-\ldots-q_{k,\nu}\cdot\lt_>(g_k)$ 
       \item $r:=r+r_\nu$ 
       \item $f_{\nu+1}:=f_\nu-q_{1,\nu}\cdot
         g_1-\ldots-q_{k,\nu}\cdot g_k-r_\nu$  
       \item $\nu:=\nu+1$
       \end{itemize}
     \end{itemize}
   \end{algorithm}

   \begin{remark}\label{rem:HDDwR}
     If $m=0$, i.e.\ if the input data $f,g_1,\ldots,g_k\in K[\ux]^s$,
     then Algorithm \ref{alg:HDDwR} terminates since for a given
     degree there are only finitely many monomials of this degree and
     therefore there cannot exist an infinite sequence of homogeneous
     polynomials $(f_\nu)_{\nu\in\N}$ of the same degree with 
     \begin{displaymath}
       \lm_>(f_1)>\lm_>(f_2)>\lm_>(f_3)>\ldots.
     \end{displaymath}
   \end{remark}


   \section{Division with Remainder in $K[[\ut]][\ux]^s$}\label{sec:DwR}

   We will use the existence of homogeneous determinate divisions
   with remainder to show that in $R[\ux]^s$ weak normal forms
   exist. In order to be able to apply this existence result we have
   to homogenise, and we need to extend our monomial ordering to the
   homogenised monomials.

   \begin{definition}
     Let $\ux_h=(x_0,\ux)=(x_0,\ldots,x_n)$.
     \begin{enumerate}
     \item      For $0\not=f\in R[\ux]^s$. We
       define the \emph{homogenisation} $f^h$ of $f$ to be
       \begin{displaymath}
         f^h:=x_0^{\deg_{\ux}(f)}\cdot
         f\left(\ut,\frac{x_1}{x_0},\ldots,\frac{x_n}{x_0}\right)
         \in R[\ux_h]_{\deg_{\ux}(f)}^s
       \end{displaymath}
       and $0^h:=0$. If $T\subset R[\ux]^s$ then we set
       \bmath
         T^h:=\big\{f^h\;\big|\;f\in T\big\}.
       \emath
     \item We call the $R[\ux]$-linear map
       \bmath
         d:R[\ux_h]^s\longrightarrow R[\ux]^s:g\mapsto g^d:=g_{|x_0=1}\lang{=g(\ut,1,\ux)}
       \emath
       the \emph{dehomogenisation} with respect to $x_0$.
     \item  Given a $\ut$-local monomial ordering $>$ on
       $\Mon^s(\ut,\ux)$ we define a $\ut$-local monomial ordering $>_h$ on
       $\Mon^s(\ut,\ux_h)$ by
       \begin{displaymath}
         \ut^\alpha\cdot\ux^\beta\cdot x_0^a\cdot
         e_i\;>_h\;\ut^{\alpha'}\cdot\ux^{\beta'}\cdot x_0^{a'}\cdot
         e_j
       \end{displaymath}
       if and only if 
       \begin{displaymath}
         |\beta|+a>|\beta'|+a'
       \end{displaymath}
       or
       \begin{displaymath}
         \big(|\beta|+a=|\beta'|+a'
         \;\mbox{ and }\;
         \ut^\alpha\cdot\ux^\beta\cdot
         e_i>\ut^{\alpha'}\cdot\ux^{\beta'}\cdot e_j\big),
       \end{displaymath}
       and we call it the \emph{homogenisation of $>$}.
     \end{enumerate}
   \end{definition}

   In the following remark we want to gather some straight forward
   properties of homogenisation and dehomogenisation.

   \begin{remark}
     Let $f,g\in R[\ux]^s$ and $F\in R[\ux_h]_k^s$. Then:
     \begin{enumerate}
     \item $f=(f^h)^d$.
     \item $F=(F^d)^h\cdot x_0^{\deg_{\ux_h}(F)-\deg_{\ux}(F^d)}$.
     \item $\lm_{>_h}(f^h)=x_0^{\ecart(f)}\cdot \lm_>(f)$.
     \item
       $\lm_{>_h}(g^h)|\lm_{>_h}(f^h)\Longleftrightarrow
       \lm_>(g)|\lm_>(f)\;\wedge\;\ecart(g)\leq\ecart(f)$.
     \item
       $\lm_{>_h}(F)=x_0^{\ecart(F^d)+\deg_{\ux_h}(F)-\deg_{\ux}(f)}
       \cdot\lm_>(F^d)$.
     \end{enumerate}
   \end{remark}

   \begin{theorem}[Division with Remainder]\label{thm:DwR}
     Let $>$ be a $\ut$-local monomial ordering on
     $\Mon^s(\ut,\ux)$ and $g_1,\ldots,g_k\in R[\ux]^s$.
     Then any $f\in R[\ux]^s$ has a weak division
     with remainder with respect to $g_1,\ldots,g_k$\kurz{.}\lang{, i.e.\ there
     exist $q_1,\ldots,q_k\in R[\ux]$, $r\in R[\ux]^s$ and $u\in S_>$ such that
     \bmath
       u\cdot f=q_1\cdot g_1+\ldots +q_k\cdot g_k+r
     \emath
     satisfies 
     \begin{enumerate}
     \item[(ID1)] $\lm_>(f)\geq \lm_>(q_i\cdot g_i)$ for
       $i=1,\ldots,k$, and
     \item[(ID2)]
       $\lm_>(r)\not\in\langle\lm_>(g_1),\ldots,\lm_>(g_k)\rangle$ 
       unless $r=0$.
     \end{enumerate}
     }
   \end{theorem}
   \begin{proof}
     The proof follows from the correctness and termination of
     Algorithm \ref{alg:DwR}, which assumes the existence of the
     homogeneous determinate division with remainder from
     Theorem \ref{thm:HDDwR} respectively Algorithm
     \ref{alg:HDDwR}.
   \end{proof}

   The following algorithm relies on the HDDwR-Algorithm, and it only
   terminates under the assumption that we are able to produce
   homogeneous determinate divisions with remainder, which implies
   that it is not an algorithm that can be applied in practise.

   \begin{algorithm}[$\DwR$ - Mora's Division with Remainder]\label{alg:DwR}
     \textsc{Input:} 
     \begin{minipage}[t]{11cm}
       $(f,G)$ with 
       $G=\{g_1,\ldots,g_k\}$ and $f,g_1,\ldots,g_k\in R[\ux]^s$, 
       $>$ a $\ut$-local monomial ordering 
     \end{minipage}
     \\[0.2cm]
     \textsc{Output:} 
     \begin{minipage}[t]{9cm}
       $(u,q_1,\ldots,q_k,r)\in S_>\times R[\ux]^k\times R[\ux]^s$ such
       that
       \begin{displaymath}
         u\cdot f=q_1\cdot g_1+\ldots +q_k\cdot g_k+r
       \end{displaymath}
       is a weak division with remainder of $f$.
     \end{minipage}
     \\[0.2cm]
     \textsc{Instructions:}
     \begin{itemize}
     \item $T:=(g_1,\ldots,g_k)$
     \item $D:=\{g_i\in T\;|\; \lm_>(g_i)\;\mbox{ divides }\;\lm_>(f)\}$
     \item IF $f\not=0$ AND $D\not=\emptyset$ DO
       \begin{itemize}
       \item IF $e:=\min\{\ecart_>(g_i)\;|\;g_i\in D\}-\ecart_>(f)>0$ THEN
         \begin{itemize}
         \item $(Q_1',\ldots,Q_k',R'):=\HDDwR\big(x_0^e\cdot
           f^h,(\lt_{>_h}(g_1^h),\ldots,\lt_{>_h}(g_k^h)\big)$
         \item $f':=\big(x_0^e\cdot f^h-\sum_{i=1}^kQ_i'\cdot g_i^h\big)^d$
         \item
           $(u'',q_1'',\ldots,q_{k+1}'',r)
           :=\DwR\big(f',(g_1,\ldots,g_k,f)\big)$
         \item $q_i:=q_i''+u''\cdot {Q_i'}^d$, $\;\;\;i=1,\ldots,k$
         \item $u:=u''-q_{k+1}''$
         \end{itemize}
       \item ELSE
         \begin{itemize}
         \item $(Q_1',\ldots,Q_k',R'):=\HDDwR\big(f^h,(g_1^h,\ldots,g_k^h)\big)$
         \item
           $(u,q_1'',\ldots,q_{k+1}'',r):=\DwR\big((R')^d,T\big)$
         \item $q_i:=q_i''+u\cdot {Q_i'}^d$, $\;\;\;i=1,\ldots,k$
         \end{itemize}
       \end{itemize}
     \item ELSE $(u,q_1,\ldots,q_k,r)=(1,0,\ldots,0,f)$
     \end{itemize}
   \end{algorithm}
   \begin{proof}
     Let us first prove the \emph{termination}. For this we denote the
     numbers, ring elements and sets, which occur in the $\nu$-th
     recursion step by a
     subscript $\nu$, e.g.\ $e_\nu$, $f_\nu$ or $T_\nu$. Since 
     \begin{displaymath}
       T_1^h\subseteq T_2^h\subseteq T_3^h\subseteq\ldots
     \end{displaymath}
     also their leading submodules in $K[\ut,\ux_h]^s$ form an ascending chain
     \begin{displaymath}
       L_{>_h}(T_1^h)\subseteq L_{>_h}(T_2^h)\subseteq L_{>_h}(T_3^h)\subseteq\ldots,
     \end{displaymath}
     and since the polynomial ring is noetherian there must be an $N$
     such that
     \begin{displaymath}
       L_{>_h}(T_\nu^h)=L_{>_h}(T_N^h)\;\;\;\forall\;\nu\geq N.
     \end{displaymath}
     If $g_{i,N}\in T_N$ such
     that $\lm_>(g_{i,N})\;|\;\lm_>(f_N)$ with
     $\ecart_>(g_{i,N})\leq\ecart_>(f_N)$, then
     \begin{displaymath}
       \lm_{>_h}(g_{i,N}^h)
       \;\big|\;
       \lm_{>_h}(f_N^h).
     \end{displaymath}
     We thus have either $\lm_{>_h}(g_{i,N}^h)\;|\;\lm_{>_h}(f_N^h)$ for some
     $g_i\in D^N\subseteq T^{N+1}$ or $f_N\in T_{N+1}$, and hence
     \begin{displaymath}
       \lm_{>_h}(f_N^h)\in L_{>_h}(T_{N+1}^h)=L_{>_h}(T_N^h).
     \end{displaymath}
     This ensures the existence of a $g_{i,N}\in T_N$ such
     that 
     \begin{displaymath}
       \lm_{>_h}(g_{i,N}^h)\;|\;\lm_{>_h}(f_N^h)
     \end{displaymath}
     which in turn implies
     that 
     \begin{displaymath}
       \lm_{>}(g_{i,N})\;|\;\lm_{>}(f_N),
     \end{displaymath}
     $e_N\leq \ecart_>(g_{i,N})-\ecart_>(f_N)\leq 0$ and $T_N=T_{N+1}$. By
     induction we conclude 
     \begin{displaymath}
       T_\nu=T_N\;\;\;\forall\;\nu\geq N,
     \end{displaymath}
     and 
     \begin{equation}\label{eq:DwR:1}
       e_{\nu}\leq 0\;\;\;\forall\;\nu\geq N.
     \end{equation}

     Since in the $N$-th recursion step we are in the first ``ELSE''
     case we have $(R'_N)^d=f_{N+1}$, and by the properties of HDDwR we
     know that for all $g\in T_N$
     \begin{displaymath}
       x_0^{\ecart_>(g)}\cdot \lm_>(g)=\lm_{>_h}(g^h)
       \;\not\big|\;
       \lm_{>_h}(R'_N)
     \end{displaymath}
     and that 
     \begin{displaymath}
       \lm_{>_h}(R'_N)
       =x_0^a\cdot\lm_{>_h}(f_{N+1}^h)
       =x_0^{a+\ecart_>(f_{N+1})}\cdot \lm_>(f_{N+1})       
     \end{displaymath}
     for some $a\geq 0$. It follows that, whenever
     $\lm_>(g)\;|\;\lm_>(f_{N+1})$, then necessarily
     \begin{equation}
       \label{eq:DwR:2}
       \ecart_>(g)> a+\ecart_>(f_{N+1})\geq\ecart_>(f_{N+1}).
     \end{equation}

     Suppose now that $f_{N+1}\not=0$ and $D_{N+1}\not=\emptyset$. Then
     we may choose $g_{i,N+1}\in D_{N+1}\subseteq T_{N+1}=T_N$ such that 
     \begin{displaymath}
       \lm_>(g_{i,N+1})\;\big|\;\lm_>(f_{N+1})
     \end{displaymath}
     and
     \begin{displaymath}
       e_{N+1}=\ecart_>(g_{i,N+1})-\ecart_>(f_{N+1}).
     \end{displaymath}
     According to \eqref{eq:DwR:1} $e_{N+1}$ is
     non-positive, while according to \eqref{eq:DwR:2} it must be
     strictly positive. Thus we have derived a contradiction which
     shows that either $f_{N+1}=0$ or $D_{N+1}=\emptyset$, and in any
     case the algorithm stops.

     Next we have to prove the \emph{correctness}. We do this by induction on
     the number of recursions, say $N$, of the algorithm. 

     If $N=1$ then either $f=0$ or $D=\emptyset$, and in both cases
     \begin{displaymath}
       1\cdot f=0\cdot g_1+\ldots+0\cdot g_k+f
     \end{displaymath}
     is a weak division with remainder of $f$ satisfying (ID1) and (ID2). 
     We may thus assume that $N>1$ and $e=\min\{\ecart_>(g)\;|\;g\in
     D\}-\ecart_>(f)$. 

     If $e\leq 0$ then by Theorem \ref{thm:HDDwR} 
     \begin{displaymath}
       f^h=Q_1'\cdot g_1^h+\ldots+Q_k'\cdot g_k^h+R'
     \end{displaymath}
     satisfies (DD1), (DD2) and (DDH). (DD1) implies that for each
     $i=1,\ldots,k$ we have 
     \begin{multline*}
       x_0^{\ecart_>(f)}\cdot \lm_>(f)=\lm_{>_h}(f^h)\geq\\ \lm_{>_h}(Q_i')\cdot
       \lm_{>_h}(g_i^h)=x_0^{a_i+\ecart_>(g_i)}\cdot \lm_>\big({Q_i'}^d\big)\cdot \lm_>(g_i)
     \end{multline*}
     for some $a_i\geq 0$,
     and since $f^h$ and $Q_i'\cdot g_i^h$ are $\ux_h$-homogeneous of the same $\ux_h$-degree 
     by (DDH) the definition of the homogenised
     ordering implies that necessarily
     \begin{displaymath}
       \lm_>(f)\geq \lm_>\big({Q_i'}^d\big)\cdot \lm_>(g_i)\;\;\;\forall\;i=1,\ldots,k.
     \end{displaymath}
     Note that
     \begin{displaymath}
       (R')^d=\left(f^h-\sum_{i=1}^k Q_i'\cdot
         g_i^h\right)^d=f-\sum_{i=1}^k{Q_i'}^d\cdot g_i,
     \end{displaymath}
     and thus
     \begin{displaymath}
       \lm_>\big((R')^d\big)=\lm_>\left(f-\sum_{i=1}^k{Q_i'}^d\cdot
         g_i\right)\leq\lm_>(f).
     \end{displaymath}
     Moreover, by induction 
     \begin{displaymath}
       u\cdot (R')^d=q_1''\cdot g_1+\ldots q_k''\cdot g_k+r
     \end{displaymath}
     satisfies (ID1) and (ID2). But (ID1) implies that
     \begin{displaymath}
       \lm_>(f)\;\geq\; \lm_>\big((R')^d\big)\;\geq\;\lm_>(q_i''\cdot g_i),
     \end{displaymath}
     so that 
     \begin{displaymath}
       u\cdot f
       =\sum_{i=1}^k \big(q_i''+u\cdot {Q_i'}^d\big)\cdot g_i+r
     \end{displaymath}
     satisfies (ID1) and (ID2).

     It remains to consider the case $e>0$. Then by Theorem \ref{thm:HDDwR} 
     \begin{equation}\label{eq:DwR:3}
       x_0^e\cdot f^h=Q_1'\cdot \lt_{>_h}(g_1^h)+\ldots+Q_k'\cdot \lt_{>_h}(g_k^h)+R'
     \end{equation}
     satisfies (DD1), (DD2) and (DDH). (DD1) and (DD2) imply (ID1) for
     this representation, which means that for some $a_i\geq 0$ 
     \begin{multline*}
       x_0^{e+\ecart_>(f)}\cdot \lm_>(f)=\lm_{>_h}(x_0^e\cdot f^h)\geq\\ 
       \lm_{>_h}(Q_i')\cdot
       \lm_{>_h}\big(\lt_{>_h}(g_i^h)\big)=x_0^{a_i+\ecart_>(g_i)}\cdot
       \lm_>({Q_i'}^d)\cdot \lm_>(g_i),
     \end{multline*}
     and since both sides are $\ux_h$-homogeneous of the same $\ux_h$-degree with
     by (DDH) we again necessarily have
     \begin{displaymath}
       \lm_>(f)\geq \lm_>\big({Q_i'}^d\big)\cdot \lm_>(g_i).
     \end{displaymath}
     Moreover, by induction
     \begin{equation}\label{eq:DwR:4}
       u''\cdot \left(f-\sum_{i=1}^k {Q_i'}^d\cdot g_i\right) =
       \sum_{i=1}^k q_i''\cdot g_i+q_{k+1}''\cdot f+r
     \end{equation}
     satisfies (ID1) and (ID2).

     Since $\lt_>(u'')=1$ we have
     \begin{displaymath}
       \lm_>(f)\geq \lm_>\big(q_i''+u''\cdot {Q_i'}^d\big)\cdot \lm_>(g_i),
     \end{displaymath}
     for $i=1,\ldots,k$ and therefore
     \begin{displaymath}
       (u''-q_{k+1}'')\cdot f=
       \sum_{i=1}^k\big(q_i''+u''\cdot
       {Q_i'}^d\big)\cdot g_i+r
     \end{displaymath}
     satisfies (ID1) and (ID2) as well. It remains to show that
     $u=u''-q_{k+1}''\in S_>$, or equivalently that
     \begin{displaymath}
       \lt_>(u''-q_{k+1}'')=1.
     \end{displaymath}
     By assumption there is a $g_i\in D$ such that
     $\lm_>(g_i)\;|\;\lm_>(f)$ and $\ecart_>(g_i)-\ecart_>(f)=e$. Therefore,
     $\lm_{>_h}(g_i^h)\;|\;x_0^e\cdot \lm_{>_h}(f^h)$ and thus in the
     representation \eqref{eq:DwR:3} the leading term of $x_0^e\cdot f^h$ has
     been cancelled by some $Q_j'\cdot \lt_{>_h}(g_j^h)$, which implies that
     \begin{displaymath}
       \lm_{>_h}(f^h)>\lm_{>_h}\left(f^h-\sum_{i=1}^k{Q_i'}\cdot
       g_i^h\right),
     \end{displaymath}
     and since both sides are $\ux_h$-homogeneous of the same
     $\ux_h$-degree, unless the right hand side is zero, we must have
     \begin{displaymath}
       \lm_>(f)>\lm_>\left(f-\sum_{i=1}^k{Q_i'}^d\cdot
       g_i\right)\geq \lm_>(q_{k+1}''\cdot f),
     \end{displaymath}
     where the latter inequality follows from (ID1) for \eqref{eq:DwR:4}.
     Thus however $\lm_>(q_{k+1}'')<1$, and since $\lm_>(u'')=1$
     we conclude that
     \begin{displaymath}
       \lt_>(u''-q_{k+1}'')=\lt_>(u'')=1.
     \end{displaymath}
     This finishes the proof.     
   \end{proof}

   \begin{remark}\label{rem:DwR}
     As we have pointed out our algorithms are not useful for
     computational purposes since Algorithm \ref{alg:HDDwR} does not in
     general terminate after a finite number of steps. If, however, the
     input data are in fact polynomials in $\ut$ and $\ux$, then we can
     replace the $t_i$ by $x_{n+i}$ and apply Algorithm \ref{alg:DwR} to
     $K[x_1,\ldots,x_{n+m}]^s$, so that it 
     terminates due to Remark \ref{rem:HDDwR} the computed weak
     division with remainder 
     \begin{displaymath}
       u\cdot f=q_1\cdot g_1+\ldots+q_k\cdot g_k+r
     \end{displaymath}
     is then \emph{polynomial} in the sense that $u,q_1,\ldots,q_k\in
     K[\ut,\ux]$ and $r\in K[\ut,\ux]^s$.
     In fact, Algorithm
     \ref{alg:DwR} is then only a variant of 
     the usual Mora algorithm.
   \end{remark}

   In the proof of Schreyer's Theorem we will need the existence of
   weak divisions with remainder satisfying (SID2)\lang{.}\kurz{, the
     proof is the same as \cite[Remark 2.3.4]{GP02}.}

   \lang{
   \begin{corollary}\label{cor:sid2}
     Let $>$ be a $\ut$-local monomial ordering on $\Mon^s(\ut,\ux)$
     and $g_1,\ldots,g_k\in R[\ux]^s$. Then any $f\in R[\ux]^s$ has a weak division
     with remainder with respect to $g_1,\ldots,g_k$ satisfying (SID2).
   \end{corollary}
   \begin{proof}
     We do the proof by induction on $s$ where for $s=1$ the condition
     (SID2) coincides with (ID2) and thus the result follows from
     Theorem \ref{thm:DwR}. We may therefore assume that $s>1$.
     
     By Theorem \ref{thm:DwR} there exists a week division
     with remainder
     \begin{equation}\label{eq:sid2:1}
       u\cdot f=q_1\cdot g_1+\ldots+q_k\cdot g_k+r,
     \end{equation}
     and obviously, there is a $j\in\{1,\ldots,s\}$ such that
     $\lm_>(r)=\lm_>(r_j\cdot e_j)$, unless $r=0$ -- in which case we
     are done. In order to keep the notation
     short we assume that $j=s$ and we may assume that the $g_i$
     are ordered in such a way that for some $1\leq l\leq k$
     \begin{displaymath}
       \lm_>(g_i)\in R[\ux]\cdot e_s
       \;\;\;\Longleftrightarrow\;\;\;
       i>l,
     \end{displaymath}
     i.e.\ only for the last $k-l$ of the $g_i$ the leading monomial
     depends on $e_s$.

     Consider now the projection 
     \begin{displaymath}
       \pi:R[\ux]^s\longrightarrow
       R[\ux]^{s-1}:(p_1,\ldots,p_s)\mapsto (p_1,\ldots,p_{s-1}),
     \end{displaymath}
     the inclusion
     \begin{displaymath}
        \iota:R[\ux]^{s-1}\longrightarrow R[\ux]^s:(p_1,\ldots,p_{s-1})\mapsto(p_1,\ldots,p_{s-1},0),
     \end{displaymath}
     and the restriction, say $>_*$, of $>$ to $\Mon^{s-1}(\ut,\ux)$ defined by
     \begin{displaymath}
       p\;>_*\;p'\;\;\;:\Longleftrightarrow\;\;\;
       \iota(p)>\iota(p')
     \end{displaymath}
     for $p,p'\in\Mon^{s-1}(\ut,\ux)$ -- which is again
     a $\ut$-local monomial ordering. Note, also that for $p\in
     R[\ux]^{s-1}$ we obviously have
     \begin{equation}\label{eq:sid2:2}
       \lm_>\big(\iota(p)\big)=\iota\big(\lm_{>_*}(p)\big).
     \end{equation}
     Moreover, due to the ordering of the $g_i$ we have for
     $i=1,\ldots,l$ 
     \begin{displaymath}
       \lm_>(g_i)=\lm_>\Big(\iota\big(\pi(g_i)\big)\Big).
     \end{displaymath}

     By induction hypothesis there exists a weak division
     with remainder of $\pi(r)=(r_1,\ldots,r_{s-1})$ with respect to $>_*$, say
     \begin{equation}\label{eq:sid2:3}
       u'\cdot \pi(r)=q_1'\cdot \pi(g_1)+\ldots+q_l'\cdot\pi(g_l)+r'
     \end{equation}
     with $u'\in S_{>_*}=S_>$, $q_1',\ldots,q_l'\in R[\ux]$ and
     $r'=(r_1',\ldots,r_{s-1}')\in R[\ux]^{s-1}$, satisfying (ID1) and
     (SID2). 

     We want to show that
     \begin{equation}\label{eq:sid2:4}
       u\cdot u'\cdot f=
       \sum_{i=1}^l(u'\cdot q_i+q_i')\cdot g_i+\sum_{i=l+1}^ku'\cdot
       q_i\cdot g_i+r'',
     \end{equation}
     with $r''=(r_1',\ldots,r_{s-1}',r_s)$, satisfies (ID1) and
     (SID2).

     Since $u,u'\in S_>$ have leading terms $1$ leading terms do not
     change by multiplication with $u$ or $u'$. Moreover, since
     \eqref{eq:sid2:1} and \eqref{eq:sid2:3} both satisfy (ID1) and
     taking \eqref{eq:sid2:2} into account we have 
     \begin{displaymath}
       \lm_>(f)\geq \lm_>(r)>\lm_>\Big(\iota\big(\pi(r)\big)\Big)\geq
       \lm_>\Big(q_i'\cdot\iota\big(\pi(g_i)\big)\Big)=\lm_>\big(q_i'\cdot g_i\big)
     \end{displaymath}
     for $i=1,\ldots,l$. Thus by (ID1) for \eqref{eq:sid2:1} we have
     \begin{displaymath}
       \lm_>(u\cdot u'\cdot f)\geq \lm_>\big((u'\cdot q_i+q_i')\cdot g_i\big),
     \end{displaymath}
     for $i=1,\ldots,l$ and 
     \begin{displaymath}
       \lm_>(u\cdot u'\cdot f)\geq \lm_>(u'\cdot q_i\cdot g_i\big)
     \end{displaymath}
     for $i=l+1,\ldots,k$,
     which shows that \eqref{eq:sid2:4} satisfies (ID1).

     Moreover, by (SID2) for \eqref{eq:sid2:3} we know that, unless $r_j'=0$,
     \begin{displaymath}
       \lm_{>_*}\big(\pi(g_i)\big)\;\not\big|\;\lm_{>_*}\big((0,\ldots,r_j',\ldots,0)\big)       
     \end{displaymath}
     for $j=1,\ldots,s-1$ and $i=1,\ldots,l$, and hence unless $r_j'=0$,
     \begin{displaymath}
       \lm_>(g_i)=\lm_>\Big(\iota\big(\pi(g_i)\big)\Big)\;\not\big|\;\lm_>(r_j'\cdot e_j)
     \end{displaymath}
     for $j=1,\ldots,s-1$ and $i=1,\ldots,l$.
     And since $\lm_>(g_i)$ involves $e_s$ for $i=l+1,\ldots,k$ but
     $r_j'\cdot e_j$ does not for $j=1,\ldots,s-1$, we have
     \begin{displaymath}
       \lm_>(g_i)\;\not\big|\;\lm_>(r_j'\cdot e_j)
     \end{displaymath}
     for any $i=1,\ldots,k$ and $j=1,\ldots,s-1$. And since by (Id2)
     for \eqref{eq:sid2:1} we also have that
     \begin{displaymath}
       \lm_>(g_i)\;\not\big|\;\lm_>(r)=\lm_>(r_s\cdot e_s)
     \end{displaymath}
     for any $i=1,\ldots,k$, we are done, i.e.\ \eqref{eq:sid2:4} satisfies (SID2).     
   \end{proof}
   }

   \begin{corollary}\label{cor:DwR}
     Let $>$ be a $\ut$-local monomial ordering on $\Mon^s(\ut,\ux)$
     and $g_1,\ldots,g_k\in R[\ux]_>^s$. Then any $f\in R[\ux]_>^s$ has a division
     with remainder with respect to $g_1,\ldots,g_k$ satisfying (SID2).     
   \end{corollary}
   \lang{
   \begin{proof}
     Let $f=\frac{f'}{u}$ and $g_i=\frac{g_i'}{u_i}$, $i=1,\ldots,k$,
     with $f',g_1',\ldots,g_k'\in R[\ux]^s$ and $u,u_1,\ldots,u_k\in
     S_>$. Consider 
     \begin{displaymath}
       v=u\cdot u_1\cdot\ldots\cdot u_k\in S_>
     \end{displaymath}
     and
     \begin{displaymath}
       f''=v\cdot f,\;g_1''=v\cdot g_1,\;\ldots,\;g_k''=v\cdot g_k\in R[\ux].
     \end{displaymath}
     By Corollary \ref{cor:sid2} there exists a weak division
     with remainder
     \begin{equation}\label{eq:corDwR}
       u''\cdot f''=q_1''\cdot g_1''+\ldots+q_k''\cdot g_k''+r'',
     \end{equation}
     satisfying (ID1) and (SID2) with $u''\in S_>$,
     $q_1'',\ldots,q_k''\in R[\ux]$ and $r\in R[\ux]^s$. Setting
     \begin{displaymath}
       q_1=\frac{q_1''}{u''},\;\ldots,\;q_k=\frac{q_k''}{u''}\in R[\ux]_>
     \end{displaymath}
     and
     \begin{displaymath}       
       r=\frac{1}{u''\cdot v}\cdot r''\in R[\ux]_>^s,
     \end{displaymath}
     then
     \begin{displaymath}
       f=q_1\cdot g_1+\ldots+q_k\cdot g_k+r
     \end{displaymath}
     and this representation satisfies (ID1) and (SID2) since
     by definition the leading monomials of the elements (including
     those of the components $r_\nu\cdot e_\nu$) involved in
     this representation are the same as those in \eqref{eq:corDwR}.
   \end{proof}
   }


   \section{Standard Bases in $K[[\ut]][\ux]^s$}\label{sec:standardbases}

   \begin{definition}
     Let $>$ be $\ut$-local monomial ordering on $\Mon^s(\ut,\ux)$,
     $I\leq R[\ux]^s$ and $J\leq R[\ux]_>^s$ be submodules. 
     \lang{

     }
     A \emph{standard basis of $I$} is a finite
     subset $G\subset I$ such that
     \bmath
       L_>(I)=L_>(G).
     \emath
     A \emph{standard basis of $J$} is a finite
     subset $G\subset J$ such that
     \bmath
       L_>(J)=L_>(G).
     \emath
     A finite subset $G\subseteq R[\ux]_>^s$ is called a
     \emph{standard basis} with respect to $>$ if $G$ is a standard
     basis of $\langle G\rangle\leq R[\ux]_>^s$.
   \end{definition}

   The existence of standard bases is immediate from Hilbert's Basis Theorem.

   \begin{proposition}\label{prop:existencestd}
     If $>$ is a $\ut$-local monomial ordering then every submodule of
     $R[\ux]^s$ and of $R[\ux]_>^s$ has a standard basis.
   \end{proposition}
   \lang{
   \begin{proof}
     Follows since $K[\ut,\ux]^s$ is noetherian.
   \end{proof}
   }

   Standard bases are so useful since they are generating sets for
   submodules of $R[\ux]_>^s$ and since
   submodule membership can be tested by division with remainder.

   \begin{proposition}\label{prop:basicstd}
     Let $>$ be $\ut$-local monomial ordering on $\Mon^s(\ut,\ux)$,
     $I,J\leq R[\ux]_>^s$ submodules, $G=(g_1,\ldots,g_k)\subset
     J$ a standard basis of $J$ and 
     $f\in R[\ux]_>^s$ with division with remainder 
     \lang{\begin{equation}\label{eq:basicstd:1}}
       \kurz{\begin{math}}
       f=q_1\cdot g_1+\ldots+q_k\cdot g_k+r.
       \kurz{\end{math}}
     \lang{\end{equation}}
     Then:
     \begin{enumerate}
     \item $f\in J$ if and only if $r=0$.
     \item $J=\langle G\rangle$.
     \item If $I\subseteq J$ and $L_>(I)=L_>(J)$, then $I=J$.
     \end{enumerate}
   \end{proposition}
   \begin{proof}
     \kurz{Word by word as in \cite[Lemma 1.6.7]{GP02}.}
     \lang{
     \begin{enumerate}
     \item If $r=0$ then obviously $f\in \langle G\rangle\subseteq J$.
       If conversely $f\in J$ then
       \begin{displaymath}
         r=f-q_1\cdot g_1-\ldots-q_k\cdot g_k\in J,
       \end{displaymath}
       and therefore $\lm_>(r)\in L_>(J)=L_>(G)$. But then (ID2)
       implies $r=0$.
     \item If $f\in J$ then by Corollary \ref{cor:DwR} $f$ has a 
       division with remainder as in \eqref{eq:basicstd:1}, but by (a)
       $r=0$ and thus $f\in \langle G\rangle$ since $u$ is a unit.
     \item By Proposition \ref{prop:existencestd} there exists a
       standard basis $G'\subseteq I\subseteq J$ of $I$. But then
       $G'$ is a standard basis of $J$, since
      \begin{displaymath}
         L_>(G')=L_>(I)=L_>(J),
       \end{displaymath}
       and thus $G'$  generates both, $I$ and $J$,  by (b).
     \end{enumerate}
     }
   \end{proof}

   In order to work, even theoretically, with standard bases it is
   vital to have a good criterion to decide whether a generating set is
   standard basis or not. In order to formulate Buchberger's Criterion
   it is helpful to have the notion of an \emph{s-polynomial}.

   \begin{definition}
     Let $>$ be a $\ut$-local monomial ordering on $R[\ux]^s$ and
     $f,g\in R[\ux]^s$. We define the \emph{s-polynomial} of $f$ and
     $g$ as
     \begin{displaymath}
       \spoly(f,g):=
       \frac{\lcm\big(\lm_>(f),\lm_>(g)\big)}{\lt_>(f)}\cdot f
       -
       \frac{\lcm\big(\lm_>(f),\lm_>(g)\big)}{\lt_>(g)}\cdot g.
     \end{displaymath}
   \end{definition}

   \begin{theorem}[Buchberger Criterion]\label{thm:buchbergercriterion}
     Let $>$ be a $\ut$-local monomial ordering on $\Mon^s(\ut,\ux)$,
     $J\leq R[\ux]_>^s$ a submodule and
     $g_1,\ldots,g_k\in J$. The following statements 
     are equivalent:
     \begin{enumerate}
     \item $G=(g_1,\ldots,g_k)$ is a standard basis of $J$.
     \item Every normal form with respect to $G$ of any element in $J$ is zero.
     \item Every element in $J$ has a  standard representation
       with respect to $G$.
     \item $J=\langle G\rangle$ and $\spoly(g_i,g_j)$ has a standard representation for
       all $i<j$.
     \end{enumerate}
   \end{theorem}
   \begin{proof}
     In Proposition \ref{prop:basicstd} we have shown that (a) implies
     (b), and the implication (b) to (c) is trivially true. And, finally, if
     $f\in J$ has a  standard representation with respect to $G$,
     then $\lm_>(f)\in L_>(G)$, so that (c) implies (a). Since
     $\spoly(g_i,g_j)\in J$ condition (d) follows from (c), and the
     hard part is to show that (d) implies actually (c). This is
     postponed to Theorem \ref{thm:schreyer}.
   \end{proof}

   Since for $G\subset R[\ux]^s$ we have $L_>\big(\langle
   G\rangle_{R[\ux]}\big)=L_>\big(\langle
   G\rangle_{R[\ux]_>}\big)$ we get the following corollary. 

   \begin{corollary}[Buchberger Criterion]\label{cor:buchbergercriterion}
     Let $>$ be a $\ut$-local monomial ordering on $\Mon^s(\ut,\ux)$ and
     $g_1,\ldots,g_k\in I\leq R[\ux]^s$. Then the following statements are
     equivalent:
     \begin{enumerate}
     \item $G=(g_1,\ldots,g_k)$ is a standard basis of $I$.
     \item Every weak normal form with respect to $G$ of any element in $I$ is zero.
     \item Every element in $I$ has a weak standard representation
       with respect to $G$.
     \item $\langle I\rangle_{R[\ux]_>}=\langle G\rangle_{R[\ux]_>}$
       and $\spoly(g_i,g_j)$ has a weak standard representation for 
       all $i<j$.
     \end{enumerate}
   \end{corollary}
   \lang{
   \begin{proof}
     If $G$ is a standard basis of $I$ then it is a standard
     basis of $J=\langle I\rangle_{R[\ux]_>}$, since $L_>(I)=L_>(J)$
     by Remark \ref{rem:leadingsubmodule}. Suppose now that
     $f\in I$ has a weak division with remainder 
     \begin{displaymath}
       u\cdot f=q_1\cdot g_1+\ldots+q_k\cdot g_k+r,
     \end{displaymath}
     then 
     \begin{displaymath}
       f=\frac{q_1}{u}\cdot g_1+\ldots+\frac{q_k}{u}\cdot
       g_k+\frac{1}{u}\cdot r
     \end{displaymath}
     is a division with remainder of $f\in I\subseteq J$, and thus
     $r=0$ by Theorem \ref{thm:buchbergercriterion}. Therefore
     (a) implies (b), and it is obvious that (b) implies (c).

     Moreover, if (c) holds and 
     $f=\frac{f'}{u'}\in J$ with $f'\in I$ and $u'\in S_>$ then by
     assumption there exists a weak standard representation
     \begin{displaymath}
       u\cdot f'=q_1\cdot g_1+\ldots+q_k\cdot g_k
     \end{displaymath}
     with $u\in S_>$ and $q_1,\ldots,q_k\in R[\ux]$. But then
     \begin{displaymath}
       f=\frac{q_1}{u\cdot u'}\cdot g_1+\ldots+\frac{q_k}{u\cdot
         u'}\cdot g_k
     \end{displaymath}
     is a standard representation of $f$, and Theorem
     \ref{thm:buchbergercriterion} implies that $G$ generates $J$ and
     that for each $i<j$ there is standard representation
     \begin{displaymath}
       \spoly(g_i,g_j)=\frac{q_1}{u_1}\cdot g_1+\ldots+\frac{q_k}{u_1}\cdot
       g_k
     \end{displaymath}
     with $q_i\in R[\ux]$ and $u_1,\ldots,u_k\in S_>$. Setting
     $u=u_1\cdots u_k\in S_>$
     and $q_i'=\frac{q_i\cdot u}{u_i}\in R[\ux]$ we get the weak
     standard representation
     \begin{displaymath}
       u\cdot \spoly(g_i,g_j)=q_1'\cdot g_1+\ldots+q_k'\cdot g_k,
     \end{displaymath}
     which shows that (d) holds true. 
     
     Finally, if (d) holds true then every weak standard
     representation
     \begin{displaymath}
       u\cdot \spoly(g_i,g_j)=q_1\cdot g_1+\ldots+q_k\cdot g_k,
     \end{displaymath}
     gives rise to a standard representation
     \begin{displaymath}
       \spoly(g_i,g_j)=\frac{q_1}{u}\cdot g_1+\ldots+\frac{q_k}{u}\cdot
       g_k,
     \end{displaymath}
     so that by Theorem \ref{thm:buchbergercriterion} $G$ is a
     standard basis of $J$. But by Remark \ref{rem:leadingsubmodule}
     $L_>(I)=L_>(J)$, so that $G$ is also a standard basis of $I$.
   \end{proof}
   }

   When working with polynomials in $\ux$ as well as in $\ut$ we can
   actually compute divisions with remainder and standard bases (see
   Remark \ref{rem:DwR}), and they are also standard bases of the
   corresponding submodules considered over $R[\ux]$ by the following
   corollary. 

   \begin{corollary}\label{cor:polynomialcase}
     Let $>$ be a $\ut$-local monomial ordering on $\Mon^s(\ut,\ux)$
     and let $G\subset K[\ut,\ux]^s$ be finite. Then $G$ is a standard
     basis of $\langle G\rangle_{K[\ut,\ux]}$ if and only if $G$ is a
     standard basis of $\langle G\rangle_{R[\ux]}$.
   \end{corollary}
   \begin{proof}
     Let $G=(g_1,\ldots,g_k)$. By Theorem \ref{thm:DwR} and Remark
     \ref{rem:DwR} each $\spoly(g_i,g_j)$ has a weak division with
     remainder with respect to $G$ such that the coefficients and
     remainders involved are polynomials in $\ux$ as well as in
     $\ut$. But by Corollary \ref{cor:buchbergercriterion} $G$ is a standard basis of
     either of $\langle G\rangle_{K[\ut,\ux]}$ and  $\langle
     G\rangle_{R[\ux]}$ if and only if all these remainders are
     actually zero.
   \end{proof}

   And thus it makes sense to formulate the classical standard basis
   algorithm also for the case $R[\ux]$.

   \begin{algorithm}[$\STD$ -- Standard Basis Algorithm]\label{alg:std}
     \textsc{Input:} 
     \begin{minipage}[t]{11cm}
       $(f_1,\ldots,f_k)\in \big(R[\ux]^s\big)^k$ and $>$ a $\ut$-local monomial ordering.
     \end{minipage}
     \\[0.2cm]
     \textsc{Output:} 
     \begin{minipage}[t]{11cm}
       $(f_1,\ldots,f_l)\in \big(R[\ux]^s\big)^l$ a standard basis of
       $\langle f_1,\ldots,f_k\rangle_{R[\ux]}$.
     \end{minipage}
     \\[0.2cm]
     \textsc{Instructions:}
     \begin{itemize}
     \item $G=(f_1,\ldots,f_k)$
     \item $P=\big((f_i,f_j)\;\big|\;1\leq i<j\leq k\big)$       
     \item WHILE $P\not=\emptyset$ DO
       \begin{itemize}
       \item Choose some pair $(f,g)\in P$
       \item $P=P\setminus \{(f,g)\}$
       \item $(u,\underline{q},r)=\DwR\big(\spoly(f,g),G)$
       \item IF $r\not=0$ THEN
         \begin{itemize}
         \item $P=P\cup \{(f,r)\;|\;f\in G\}$
         \item $G=G\cup \{r\}$
         \end{itemize}
       \end{itemize}
     \end{itemize}     
   \end{algorithm}
   \lang{
   \begin{proof}
     Since in each step when $G$ is enlarged the leading module of $G$
     is strictly enlarged and since $K[\ut,\ux]^s$ is noetherian the
     algorithm will terminate. Moreover, by Buchberger's Criterion $G$
     will be a standard basis.
   \end{proof}
   }

   \begin{remark}\label{rem:polynomialcase}
     If the input of $\STD$ are polynomials in $K[\ut,\ux]$ then the
     algorithm works in practise due to Remark \ref{rem:DwR}, and it
     computes a standard basis $G$ of $\langle
     f_1,\ldots,f_k\rangle_{K[\ut,\ux]}$ which due to Corollary
     \ref{cor:polynomialcase} is also a standard basis of $\langle
     f_1,\ldots,f_k\rangle_{R[\ux]}$, since $G$ still contains the
     generators $f_1,\ldots,f_k$.
     \kurz{

     Having division with remainder, standard bases and Buchberger's
     Criterion at hand one can, from a theoretical point of view,
     basically derive all the standard algorithms from computer algebra
     also for free modules over
     $R[\ux]$ respectively $R[\ux]_>$. Moreover, if the input is
     polynomial in $\ut$ and $\ux$, then the corresponding operations
     computed over $K[\ut,\ux]_>$ will
     also lead to generating sets for the corresponding operations   
     over $R[\ux]_>$.
     }
   \end{remark}


   \section{Schreyer's Theorem for $K[[t_1,\ldots,t_m]][x_1,\ldots,x_n]^s$}\label{sec:schreyer}

   In this section we want to prove Schreyer's Theorem for
   $R[\ux]^s$ which proves Buchberger's
   Criterion and shows at the same time that a standard basis of a
   submodule gives rise to a standard basis of the syzygy module
   defined by it with respect to a special ordering.
   
   \begin{definition}[Schreyer Ordering]
     Let $>$ be a $\ut$-local monomial ordering on $\Mon^s(\ut,\ux)$
     and $g_1,\ldots,g_k\in R[\ux]_>^s$. We define a \emph{Schreyer
       ordering} with respect to $>$ and $(g_1,\ldots,g_k)$, say
     $>_S$, on $\Mon^k(\ut,\ux)$ by 
     \begin{displaymath}
       \ut^\alpha\cdot\ux^\beta\cdot \varepsilon_i\;>_S\;
       \ut^{\alpha'}\cdot\ux^{\beta'}\cdot \varepsilon_j
     \end{displaymath}
     if and only if 
     \begin{displaymath}
       \ut^\alpha\cdot\ux^\beta\cdot\lm_>(g_i)\;>\;
       \ut^{\alpha'}\cdot\ux^{\beta'}\cdot\lm_>(g_j)
     \end{displaymath}
     or 
     \begin{displaymath}
       \ut^\alpha\cdot\ux^\beta\cdot\lm_>(g_i)\;=\;
       \ut^{\alpha'}\cdot\ux^{\beta'}\cdot\lm_>(g_j)
       \;\mbox{ and }\;
       i<j,
     \end{displaymath}
     where $\varepsilon_i=(\delta_{ij})_{j=1,\ldots,k}$ is the canonical basis
     with $i$-th entry one and the rest zero.

     Moreover, we define the \emph{syzygy module of $(g_1,\ldots,g_k)$}
     to be
     \begin{displaymath}
       \syz(g_1,\ldots,g_k):=
       \{(q_1,\ldots,q_k)\in R[\ux]_>^k\;|\;q_1\cdot g_1+\ldots+q_k\cdot
       g_k=0\},
     \end{displaymath}
     and we call the elements of $\syz(g_1,\ldots,g_k)$
     \emph{syzygies} of $g_1,\ldots,g_k$.
   \end{definition}

   \begin{remark}\label{rem:schreyer}
     Let $>$ be a $\ut$-local monomial ordering on $\Mon^s(\ut,\ux)$
     and $g_1,\ldots,g_k\in R[\ux]_>^s$. Let us fix for each $i<j$ a 
     division with remainder of $\spoly(g_i,g_j)$, say
     \begin{equation}\label{eq:schreyer:1}
       \spoly(g_i,g_j)=\sum_{\nu=1}^k q_{i,j,\nu}\cdot g_\nu+r_{ij},
     \end{equation}
     and define
     \begin{displaymath}
       m_{ji}:=\frac{\lcm\big(\lm_>(g_i),\lm_>(g_j)\big)}{\lm_>(g_i)},
     \end{displaymath}
     so that 
     \begin{displaymath}
       \spoly(g_i,g_j)=\frac{m_{ji}}{\lc_>(g_i)}\cdot g_i-\frac{m_{ij}}{\lc_>(g_j)}\cdot g_j.
     \end{displaymath}
     Then
     \begin{displaymath}
       s_{ij}:=\frac{m_{ji}}{\lc_>(g_i)}\cdot\varepsilon_i-
       \frac{m_{ij}}{\lc_>(g_j)}\cdot\varepsilon_j-
       \sum_{\nu=1}^k q_{i,j,\nu}\cdot \varepsilon_\nu\in R[\ux]_>^k
     \end{displaymath}
     has the property
     \begin{displaymath}
       s_{ij}\in\syz(g_1,\ldots,g_k)\;\;\;\Longleftrightarrow\;\;\; r_{ij}=0.
     \end{displaymath}
     \lang{
     Moreover, the leading monomial of $s_{ij}$ with respect to the
     Schreyer ordering on $\Mon^k(\ut,\ux)$ induced by $>$ and
     $(g_1,\ldots,g_k)$ is
     \begin{equation}\label{eq:schreyer:2}
       \lm_{>_S}(s_{ij})=m_{ji}\cdot \varepsilon_i,
     \end{equation}
     since $\lm_>\big(\frac{m_{ji}}{\lc_>(g_i)}\cdot g_i\big)
     =\lm_>\big(\frac{m_{ij}}{\lc_>(g_j)}\cdot g_j\big)$ but $i<j$
     and since \eqref{eq:schreyer:1} satisfies (ID1).     
     }
   \end{remark}

   \begin{theorem}[Schreyer]\label{thm:schreyer}
     Let $>$ be a $\ut$-local monomial ordering on $\Mon^s(\ut,\ux)$,
     $g_1,\ldots,g_k\in R[\ux]_>^s$ and suppose that
     $\spoly(g_i,g_j)$ has a weak standard representation with respect
     to $G=(g_1,\ldots,g_k)$ for each $i<j$. 

     Then $G$ is a standard basis,
     and
     with the notation in
     Remark \ref{rem:schreyer}  $\{s_{ij}\;|\;i<j\}$ is a standard basis of
     $\syz(g_1,\ldots,g_k)$ with respect to $>_S$. 
   \end{theorem}
   \begin{proof}
     \kurz{The same as in \cite[Theorem 2.5.9]{GP02}.}
     \lang{
     Let $I=\langle G\rangle_{R[\ux]_>}$ and consider the
     $R[\ux]_>$-linear map
     \begin{displaymath}
       \phi:R[\ux]_>^k\longrightarrow
       R[\ux]_>^s:(q_1,\ldots,q_k)
       \mapsto q_1\cdot g_1+\ldots+q_k\cdot g_k.
     \end{displaymath}
     For $f\in I$ there is a $q:=(q_1,\ldots,q_k)\in R[\ux]_>^k$ such
     that $f=\sum_{i=1}^kq_i\cdot g_i$, and by Corollary \ref{cor:DwR}
     there is a  division with remainder of $q$ with respect to
     $(s_{ij}\;|\;i<j)$ and $>_S$, say 
     \begin{equation}\label{eq:schreyer:3}
       q=\sum_{i<j} a_{ij}\cdot s_{ij}+r
     \end{equation}
     with $a_{ij}\in R[\ux]_>$ and
     $r=\sum_{\nu=1}^kr_\nu\cdot \varepsilon_\nu\in R[\ux]_>^k$, 
     which satisfies (ID1) and (SID2). By (SID2)
     \begin{displaymath}
       m_{ji}\cdot \varepsilon_i=\lm_{>_S}(s_{ij})\;\not\big|\;\lm_{>_S}(r_\nu\cdot\varepsilon_\nu)
     \end{displaymath}
     whenever $r_\nu\not=0$, and hence
     \begin{equation}\label{eq:schreyer:6}
       m_{j\nu}\;\not\big|\;r_\nu,
     \end{equation}
     whenever $r_\nu\not=0$. 

     Note that 
     \begin{equation}\label{eq:schreyer:4}
       f=\phi(q)=\phi(r)=\sum_{\nu=1}^k r_\nu\cdot g_\nu,
     \end{equation}
     since $s_{ij}\in\ker(\phi)$, and we claim that 
     \begin{equation}\label{eq:schreyer:5}
       \lm_>(f)\geq \lm_>(r_\nu\cdot g_\nu). 
     \end{equation}
     For this it suffices to show that 
     \begin{displaymath}
       \lm_>(r_\nu)\cdot \lm_>(g_\nu)\not=\lm_>(r_\mu)\cdot \lm_>(g_\mu)
     \end{displaymath}
     for $\nu<\mu$, whenever $r_\nu\not=0\not=r_\mu$. Suppose the contrary, then
     \begin{displaymath}
       0\not=m_{\mu\nu}\cdot\lm_>(g_\nu)=\lcm\big(\lm_>(g_\nu),\lm_>(g_\mu)\big)
     \end{displaymath}
     divides
     \begin{displaymath}
       \lm_>(r_\mu)\cdot \lm_>(g_\mu)=\lm_>(r_\nu)\cdot \lm_>(g_\nu),       
     \end{displaymath}
     since both $\lm_>(g_\nu)$ and $\lm_>(g_\mu)$ divide the latter.
     But this contradicts \eqref{eq:schreyer:6}.

     It follows from \eqref{eq:schreyer:5}
     that \eqref{eq:schreyer:4}  is a standard representation of $f$ with respect to
     $(g_1,\ldots,g_k)$ and $>$, and since $f\in I$ was arbitrary it
     follows from Theorem \ref{thm:buchbergercriterion}
     ``(c)$\Longrightarrow$(a)'', which we have already proved, that
     $G$ is actually a standard basis of $\langle
     G\rangle_{R[\ux]_>}$.
     
     Moreover, $q\in\syz(g_1,\ldots,g_k)$ if and only if $\phi(q)=f=0$, and
     by \eqref{eq:schreyer:4} and \eqref{eq:schreyer:5} this is the
     case if and only if $r=0$. Thus by \eqref{eq:schreyer:3} every
     element in $\syz(g_1,\ldots,g_k)$ has a standard
     representation  with respect to $\{s_{ij}\;|\;i<j\}$ and $>_S$,
     and therefore, as before, $\{s_{ij}\;|\;i<j\}$ is a standard
     basis of $\syz(g_1,\ldots,g_k)$ with respect to $>_S$ by 
     Theorem \ref{thm:buchbergercriterion}
     ``(c)$\Longrightarrow$(a)''. This finishes the proof.     
     }
   \end{proof}


   \lang{ 

   \section{Algorithms Relying on Standard Bases}\label{sec:algorithms}

   Having division with remainder, standard bases and Buchberger's
   Criterion at hand one can, from a theoretical point of view,
   basically derive all the standard results from computer algebra
   also for free modules over
   $R[\ux]$ respectively $R[\ux]_>$. We will gather here some of
   these results which are explicitly needed for the Lifting
   Algorithm for tropical varieties.

   The simplest algorithm is the one for testing submodule membership.

   \begin{algorithm}[$\MEMBERSHIP$]\label{alg:membership}
     \textsc{Input:} 
     \begin{minipage}[t]{11cm}
       $f,f_1,\ldots,f_k\in R[\ux]^s$ and $>$ a
       $\ut$-local monomial ordering.
     \end{minipage}
     \\[0.2cm]
     \textsc{Output:} 
     \begin{minipage}[t]{11cm}
       $N$, where $N=1$ if $f\in \langle f_1,\ldots,f_k\rangle_{R[\ux]_>}$, and $N=0$
       else.
     \end{minipage}
     \\[0.2cm]
     \textsc{Instructions:}
     \begin{itemize}
     \item $(u,\underline{q},r)=\DwR\big(f,(f_1,\ldots,f_k),>\big)$
     \item IF $r=0$ THEN $N=1$ ELSE $N=0$
     \end{itemize}          
   \end{algorithm}

   In order to do more complicated computations one needs elimination.

   \begin{definition}
     Divide the variables $\ux=(x_1,\ldots,x_n)$ into two disjoint
     subsets $\ux_0$ and $\ux_1$. We call a $\ut$-local monomial
     ordering $>$ on $\Mon^s(\ut,\ux)$ an \emph{elimination ordering
     with respect to $\ux_1$} if for $f\in R[\ux]$
     \begin{displaymath}
       \lm_>(f)\in K[\ut,\ux_1]^s\;\;\;\Longrightarrow\;\;\;
       f\in R[\ux_1]^s.
     \end{displaymath}
     Typical examples of elimination orderings are \emph{block
       orderings} like the one defined by
     \begin{displaymath}
       \ut^\alpha\cdot {\ux_0}^{\beta}\cdot{\ux_1}^\gamma\cdot e_i
       \;>\;
       \ut^{\alpha'}\cdot {\ux_0}^{\beta'}\cdot{\ux_1}^{\gamma'}\cdot e_j      
     \end{displaymath}
     if and only if
     \begin{displaymath}
       {\ux_0}^{\beta}\;>_0\;{\ux_0}^{\beta'}
     \end{displaymath}
     or 
     \begin{displaymath}
       {\ux_0}^{\beta}\;=\;{\ux_0}^{\beta'}\;\mbox{ and }\;
       \ut^\alpha\cdot{\ux_1}^\gamma\cdot e_i\;>_1\;
       \ut^{\alpha'}\cdot{\ux_1}^{\gamma'}\cdot e_j,
     \end{displaymath}
     where $>_0$ is a global monomial ordering on $\Mon(\ux_0)$ and
     $>_1$ is a $\ut$-local monomial ordering on
     $\Mon^s(\ut,\ux_1)$. Denote $>$ by $(>_0,>_1)$.
   \end{definition}

   \begin{proposition}
     Let $>$ be a $\ut$-local elimination
     ordering with respect to $\ux_0$ on $\Mon^s(\ut,\ux)$,
     $I\leq R[\ux]_>^s$, and $G$
     be a standard basis of $I$ with respect to $>$.

     Then $\big(g\in G\;\big|\;\lm_>(g)\in K[t,\ux\setminus\ux_0]^s\big)$ is a standard basis of
     $I\cap R[\ux\setminus\ux_0]_>^s$.
   \end{proposition}
   \begin{proof}
     $G'=\big(g\in G\;\big|\;\lm_>(g)\in K[t,\ux\setminus\ux_0]^s\big)$ 
     is contained in $I\cap R[\ux\setminus\ux_0]^s$ since $>$ is an 
     elimination ordering with respect to
     $\ux_0$. Moreover, if $f\in I\cap R[\ux\setminus\ux_0]_>^s\subseteq I$ then there
     is there is a $g\in G$ such that
     $\lm_>(g)\;|\;\lm_>(f)\in K[t,\ux\setminus\ux_0]^s$, since $G$ is a standard
     basis of $I$. However, this forces $\lm_>(g)\in K[t,\ux\setminus\ux_0]^s$ and
     thus $g\in G'$. This shows that $G'$ is a standard basis of
     $I\cap R[\ux\setminus\ux_0]_>^s$.
   \end{proof}

   This leads to the following elimination algorithm.

   \begin{algorithm}[$\ELIMINATE$]\label{alg:eliminate}
     \textsc{Input:} 
     \begin{minipage}[t]{11cm}
       $f_1,\ldots,f_k\in R[\ux]_>^s$, $\ux_0\subseteq \ux$, and $>$ a
       $\ut$-local elimination ordering with respect to $\ux_0$.
     \end{minipage}
     \\[0.2cm]
     \textsc{Output:} 
     \begin{minipage}[t]{11cm}
       $G\subset R[\ux\setminus\ux_0]^s$ a standard basis of
       $\langle f_1,\ldots,f_k\rangle_{R[\ux]_>}\cap R[\ux\setminus\ux_0]_>$.
     \end{minipage}
     \\[0.2cm]
     \textsc{Instructions:}
     \begin{itemize}
     \item $G'=\STD(f_1,\ldots,f_k,>)$
     \item $G=\big(g\in G\;\big|\;\lm_>(g)\in K[\ut,\ux\setminus\ux_0]^s\big)$
     \end{itemize}     
   \end{algorithm}

   \begin{proposition}
     Let $>$ be a $\ut$-local
     monomial ordering on $\Mon^s(\ut,\ux)$,
     and let $I=\langle f_1,\ldots,f_k\rangle, J=\langle
     g_1,\ldots,g_l\rangle\leq R[\ux]_>^s$, then
     \begin{displaymath}
       I\cap J=\langle \tau\cdot f_1,\ldots,\tau\cdot f_k,(1-\tau)\cdot
       g_1,\ldots,(1-\tau)\cdot g_l\rangle_{R[\ux]_>[\tau]}\cap R[\ux]_>^s. 
     \end{displaymath}
   \end{proposition}
   \begin{proof}
     If $f=\sum_{i=1}^ka_i\cdot f_i=\sum_{j=1}^lb_j\cdot g_j$ with
     $a_i,b_j\in R[\ux]_>$ then
     \begin{displaymath}
       f=\tau\cdot f+(1-\tau)\cdot f=\sum_{i=1}^ka_i\cdot \tau\cdot f_i
       +\sum_{j=1}^lb_j\cdot (1-\tau)\cdot g_j
     \end{displaymath}
     is in the right hand side. Conversely, if 
     \begin{displaymath}
       f=\sum_{i=1}^ka_i\cdot \tau\cdot f_i
       +\sum_{j=1}^lb_j\cdot (1-\tau)\cdot g_j
     \end{displaymath}
     is in the right hand side with $a_i,b_j\in R[\ux]_>[\tau]$, then
     \begin{displaymath}
       f=f_{|\tau=0}=\sum_{i=1}^k a_{i|\tau=0}\cdot f_i\in I
     \end{displaymath}
     and
     \begin{displaymath}
       f=f_{|\tau=1}=\sum_{j=1}^l b_{j|\tau=1}\cdot g_j\in J,
     \end{displaymath}
     so that $f\in I\cap J$.
   \end{proof}
   
   This leads to the following algorithm for computing intersections
   of submodules.
   
   \begin{algorithm}[$\INTERSECTION$]\label{alg:intersection}
     \textsc{Input:} 
     \begin{minipage}[t]{11cm}
       $f_1,\ldots,f_k,g_1,\ldots,g_l\in R[\ux]^s$ and $>$ a
       $\ut$-local ordering.
     \end{minipage}
     \\[0.2cm]
     \textsc{Output:} 
     \begin{minipage}[t]{11cm}
       $G\subset R[\ux]^s$ a standard basis of
       $\langle f_1,\ldots,f_k\rangle_{R[\ux]_>}\cap
       \langle g_1,\ldots,g_l\rangle_{R[\ux]_>}$.
     \end{minipage}
     \\[0.2cm]
     \textsc{Instructions:}
     \begin{itemize}
     \item Let $>'=(>_0,>)$ be the block ordering with respect to the
       unique global ordering $>_0$ on $\Mon(\tau)$
     \item $G=\hspace*{-0.05cm}\ELIMINATE\Big(\big(\tau f_1,\ldots,\tau
       f_k,(1-\tau) g_1,\ldots,(1-\tau) g_l\big),\tau,>'\Big)$
     \end{itemize}          
   \end{algorithm}

   \begin{proposition}
     Let $>$ be a $\ut$-local monomial ordering on $\Mon^s(\ut,\ux)$,
     let $I=\langle f_1,\ldots,f_k\rangle \leq R[\ux]_>^s$ and $0\not=f\in
     R[\ux]_>^s$. 

     If $I\cap \langle f\rangle=\langle g_1\cdot
     f,\ldots,g_l\cdot f\rangle$, then $I:\langle f\rangle=\langle
     g_1,\ldots,g_l\rangle_{R[\ux]_>}$.
   \end{proposition}
   \begin{proof}
     It is clear that $g_1,\ldots,g_l\in I:\langle f\rangle$, and we
     may thus suppose that we have some $h\in I:\langle f\rangle\unlhd
     R[\ux]_>$. By assumption $h\cdot f\in I\cap \langle f\rangle$ and
     thus there are $a_1,\ldots,a_k\in R[\ux]_>$ such that
     \begin{displaymath}
       h\cdot f=\left(\sum_{i=1}^k a_i\cdot g_i\right)\cdot f,
     \end{displaymath}
     and since $R[\ux]_>$ has no zero divisors this implies
     $h=\sum_{i=1}^k a_i\cdot g_i$. 
   \end{proof}

   We thus get the following algorithm for computing the ideal
   quotient with respect to a single element.

   \begin{algorithm}[$\QUOTIENT$]\label{alg:quotient}
     \textsc{Input:} 
     \begin{minipage}[t]{11cm}
       $f_1,\ldots,f_k,f\in R[\ux]^s$ and $>$ a
       $\ut$-local monomial ordering.
     \end{minipage}
     \\[0.2cm]
     \textsc{Output:} 
     \begin{minipage}[t]{11cm}
       $G\subset R[\ux]$ a standard basis of
       $\langle f_1,\ldots,f_k\rangle_{R[\ux]_>}:\langle f\rangle_{R[\ux]_>}$.
     \end{minipage}
     \\[0.2cm]
     \textsc{Instructions:}
     \begin{itemize}
     \item $G'=\INTERSECTION\big((f_1,\ldots,f_k),(f)\big)$
     \item $G=\big(\frac{g}{f}\;\big|\;g\in G\big)$
     \end{itemize}          
   \end{algorithm}

   Finally, this leads to the following algorithm for computing the
   saturation with respect to a single element.

   \begin{algorithm}[$\SATURATION$]\label{alg:saturation}
     \textsc{Input:} 
     \begin{minipage}[t]{11cm}
       $f_1,\ldots,f_k,f\in R[\ux]^s$ and $>$ a
       $\ut$-local monomial ordering.
     \end{minipage}
     \\[0.2cm]
     \textsc{Output:} 
     \begin{minipage}[t]{11cm}
       $G\subset R[\ux]$ a standard basis of
       $\langle f_1,\ldots,f_k\rangle_{R[\ux]_>}:\langle f\rangle^\infty_{R[\ux]_>}$.
     \end{minipage}
     \\[0.2cm]
     \textsc{Instructions:}
     \begin{itemize}
     \item $N=0$
     \item WHILE $N=0$ DO
       \begin{itemize}
       \item $G'=\QUOTIENT\big(G,f,>\big)$
       \item Test using $\MEMBERSHIP$ if $G'\subset\langle G\rangle$.
       \item IF so THEN $N=1$ 
       \end{itemize}
     \end{itemize}          
   \end{algorithm}
   \begin{proof}
     The procedure constructs an ascending sequence of modules
     generated by the sets $G$, and since $R[\ux]_>$ is noetherian the
     algorithm must stop after a finite number of steps. Moreover,
     once the procedure stops then 
     \begin{displaymath}
       \langle G\rangle :\langle f\rangle =\langle G'\rangle =\langle G\rangle,
     \end{displaymath}
     which shows that $\langle G\rangle$ is actually saturated with
     respect to $f$.
   \end{proof}

   We will use the existence of these procedures at the end of the
   next section to show that generators for $\langle
   f_1,\ldots,f_k\rangle_{R[\ux]}:\langle t\rangle^\infty_{R[\ux]}$
   can be computed over $K[t,\ux]$ when the $f_i\in K[t,\ux]$ are
   polynomials. 

   }


   \section{Application to $t$-Initial Ideals}\label{sec:application}

   In this section we want to show that for an ideal $J$ over the 
   field of Puiseux series which is generated by elements in
   $K[[t^\frac{1}{N}]][\ux]$ respectively in $K[t^\frac{1}{N},\ux]$ the
   $t$-initial ideal (a notion we will introduce further down)
   with respect to $w\in \Q_{<0}\times\Q^n$   can be
   computed from a standard basis of the generators. 

   \begin{definition}
     We consider for $0\not=N\in \N$ the discrete valuation ring
     \begin{displaymath}
       R_N\big[\big[t^\frac{1}{N}\big]\big]=
       \left\{\sum_{\alpha=0}^\infty a_\alpha\cdot t^\frac{\alpha}{N}\;\big|\;a_\alpha\in K\right\}
     \end{displaymath}
     of power series in the unknown $t^\frac{1}{N}$
     with \emph{discrete valuation}
     \begin{displaymath}
       \val\left(\sum_{\alpha=0}^\infty a_\alpha\cdot
       t^\frac{\alpha}{N}\right)
       =\ord_t\left(\sum_{\alpha=0}^\infty a_\alpha\cdot
       t^\frac{\alpha}{N}\right)=
       \min\left\{\frac{\alpha}{N}\;\Big|\;a_\alpha\not=0\right\}\in\frac{1}{N}\cdot\Z,
     \end{displaymath}
     and we denote by
     \bmath
       L_N=\Quot(R_N)
     \emath
     its quotient field. If $N\;|\;M$ then
     in an obvious way we can think of $R_N$ as a subring of $R_M$, and
     thus of $L_N$ as a subfield of $L_M$.
     We call the direct limit of the corresponding direct system
     \begin{displaymath}
       L=K\{\{t\}\}=\lim_{\longrightarrow} L_N=\bigcup_{N\geq 0}L_N
     \end{displaymath}
     the \emph{field of (formal) Puiseux series} over $K$. 
   \end{definition}

   \begin{remark}
     If $0\not=N\in\N$ then
     \bmath
       S_N=\lang{\big}\{1,t^\frac{1}{N},t^\frac{2}{N},t^\frac{2}{N},\ldots\lang{\big}\}
     \emath
     is a multiplicative subset of $R_N$, and obviously
     \bmath
       L_N=S_N^{-1}R_N=\lang{\left}\{t^\frac{-\alpha}{N}\cdot f
         \;\lang{\bigg}|\;f\in R_N,\alpha\in\N\lang{\right}\},
     \emath
     since
     \bmath
       R_N^*=\lang{\left}\{\sum_{\alpha=0}^\infty a_\alpha\cdot t^\frac{\alpha}{N}
         \;\lang{\bigg}|\;a_0\not=0\lang{\right}\}.
     \emath
     The valuations of $R_N$ extend to $L_N$, and thus $L$, by
     \bmath
       \val\lang{\left}\kurz{\big}(\frac{f}{g}\lang{\right}\kurz{\big})=\val(f)-\val(g)
     \emath
     for $f,g\in R_N$ with $g\not=0$.
   \end{remark}

   \begin{definition}
     For $0\not=N\in\N$ if we consider $t^\frac{1}{N}$ as a variable,
     we get the set of monomials 
     \bmath
       \Mon\big(t^{\frac{1}{N}},\ux\big)=\left\{t^\frac{\alpha}{N}\cdot \ux^\beta\; \big|\;
           \alpha\in\N,\beta\in\N^n\right\}
     \emath
     in $t^\frac{1}{N}$ and $\ux$. If $N\;|\;M$ then obviously 
     \bmath
       \Mon\big(t^{\frac{1}{N}},\ux\big)\subset\Mon\big(t^{\frac{1}{M}},\ux\big).
     \emath
   \end{definition}

   \begin{remarkdefinition}
     Let $0\not=N\in\N$, $w=(w_0,\ldots,w_n)\in
     \R_{<0}\times\R^n$, and $q\in\R$.

     We may consider the direct product
       \begin{displaymath}
         V_{q,w,N}=\prod_{\tiny\begin{array}{c}(\alpha,\beta)\in\N^{n+1}
             \\w\cdot(\frac{\alpha}{N},\beta)=q\end{array}}
         K\cdot t^\frac{\alpha}{N}\cdot \ux^\beta
       \end{displaymath}
       of $K$-vector spaces and its subspace
       \begin{displaymath}
         W_{q,w,N}=\bigoplus_{\tiny\begin{array}{c}(\alpha,\beta)\in\N^{n+1}
             \\w\cdot(\frac{\alpha}{N},\beta)=q\end{array}} K\cdot
         t^\frac{\alpha}{N}\cdot \ux^\beta.
     \end{displaymath}
     As a $K$-vector space the formal power series ring $K\big[\big[t^\frac{1}{N},\ux\big]\big]$ is just
     \begin{displaymath}
       K\big[\big[t^\frac{1}{N},\ux\big]\big]=\prod_{q\in\R}V_{q,w,N},
     \end{displaymath}
     and we can thus write any power series $f\in K\big[\big[t^\frac{1}{N},\ux\big]\big]$ in a
     unique way as
     \begin{displaymath}
       f=\sum_{q\in\R}f_{q,w} \;\;\;\mbox{ with }\;\;\;
       f_{q,w}\in V_{q,w,N}.
     \end{displaymath}   
     Note that this representation is independent of $N$ in the sense
     that if $f\in K\big[\big[t^\frac{1}{N'},\ux\big]\big]$ for some other
     $0\not=N'\in\N$ then we get the same non-vanishing $f_{q,w}$ if
     we decompose $f$ with respect to $N'$. 
     
     Moreover, if $0\not=f\in R_N[\ux]\subset K\big[\big[t^\frac{1}{N},\ux\big]\big]$, then there is
     a maximal $\hat{q}\in \R$ such that $f_{\hat{q},w}\not=0$ and 
     \bmath
       f_{q,w}\in W_{q,w,N}\lang{\;\;\;}\mbox{ for all
       }\lang{\;\;\;}q\in\R,
     \emath
     since the $\ux$-degree of the monomials involved in $f$ is
     bounded. We call the elements $f_{q,w}$ 
     \emph{$w$-quasihomogeneous} of $w$-degree $\deg_w(f_{q,w})=q\in \R$, 
     \begin{displaymath}
       \IN_w(f)=f_{\hat{q},w}\in K\big[t^\frac{1}{N},\ux\big]
     \end{displaymath}
     the \emph{$w$-initial form} of $f$ or the \emph{initial form of $f$
       w.r.t.\ $w$}, and 
     \begin{displaymath}
       \ord_w(f)=\hat{q}=\max\{\deg_w(f_{q,w})\;|\;f_{q,w}\not=0\}
     \end{displaymath}
     the \emph{$w$-order} of $f$.
     \lang{

     }
     For $I\subseteq R_N[\ux]$ we call
     \begin{displaymath}
       \IN_w(I)=\big\langle \IN_w(f)\;\big|\;f\in I\big\rangle
       \unlhd K\big[t^\frac{1}{N},\ux\big]
     \end{displaymath}
     the \emph{$w$-initial ideal} of $I$. Note that its definition
     depends on $N$!

     Moreover, we call
     \begin{displaymath}
       \tin_w(f)=\IN_w(f)(1,\ux)=\IN_w(f)_{|t=1}\in K[\ux]
     \end{displaymath}
     the \emph{$t$-initial form of $f$ w.r.t.\ $w$}, and if
     $f=t^\frac{-\alpha}{N}\cdot g\in L[\ux]$ with $g\in
     R_N[\ux]$ we set
     \bmath
       \tin_w(f):=\tin_w(g).
     \emath
     This definition does not depend on the particular representation of
     $f$\kurz{.}\lang{, since $t^\frac{-\alpha}{N}\cdot g=t^\frac{-\alpha'}{N'}\cdot
     g'$ implies that $t^\frac{\alpha'}{N'}\cdot
     g=t^\frac{\alpha}{N}\cdot g'$ in $R_{N\cdot N'}$ and thus 
     \begin{displaymath}
       t^\frac{\alpha'}{N'}\cdot\IN_w(g)=\IN_w\big(t^\frac{\alpha'}{N'}\cdot
       g\big)=\IN_w\big(t^\frac{\alpha}{N}\cdot g'\big)=t^\frac{\alpha}{N}\cdot\IN_w(g'),
     \end{displaymath}
     which shows that $\tin_w(g)=\tin_w(g')$. 

     }
     If $I\subseteq L[\ux]$
     is an ideal, then 
     \begin{displaymath}
       \tin_w(I)=\langle \tin_w(f)\;|\;f\in I\rangle\lhd K[\ux]
     \end{displaymath}
     is the \emph{$t$-initial ideal} of $I$, which does not depend on
     any $N$.

     Note also that the product of two $w$-quasihomogeneous elements 
     \bmath
       f_{q,w}\cdot f_{q',w}\in V_{q+q',w,N},
     \emath
     \lang{
     and thus
     \begin{displaymath}
       W_{q,w,N}\cdot W_{q,w,N}
       \subseteq W_{q+q',w,N}
     \end{displaymath}
     and
     \begin{displaymath}
       V_{q,w,N}\cdot V_{q',w,N}\subseteq V_{q+q',w,N}.    
     \end{displaymath}
     I}\kurz{and i}n particular, 
     \bmath
       \IN_w(f\cdot g)=\IN_w(f)\cdot\IN_w(g)
     \emath
     for $f,g\in R_N[\ux]$, and for $f,g\in L[\ux]$
     \bmath
       \tin_w(f\cdot g)=\tin_w(f)\cdot\tin_w(g).
     \emath
     An immediate consequence of this is the following lemma.
   \end{remarkdefinition}

   \begin{lemma}\label{lem:initialform}
     If $0\not=f=\sum_{i=1}^k g_i\cdot h_i$ with $f,g_i,h_i\in R_N[\ux]$ and
     $\ord_w(f)\geq \ord_w(g_i\cdot h_i)$ for all $i=1,\ldots,k$, then
     \begin{displaymath}
       \IN_w(f)\in\big\langle \IN_w(g_1),\ldots,\IN_w(g_k)\big\rangle\lhd K\big[t^\frac{1}{N},\ux\big].
     \end{displaymath}
   \end{lemma}
   \begin{proof}
     Due to the direct product decomposition we have that
     \begin{displaymath}
       \IN_w(f)=f_{\hat{q},w}=\sum_{i=1}^k (g_i\cdot h_i)_{\hat{q},w}
     \end{displaymath}
     where $\hat{q}=\ord_w(f)$. By assumption $\ord_w(g_i)+\ord_w(h_i)=\ord_w(g_i\cdot
     h_i)\leq\ord_w(f)=\hat{q}$ with equality if
     and only if $(g_i\cdot h_i)_{\hat{q},w}\not=0$. In that case
     necessarily 
     \bmath
       (g_i\cdot h_i)_{\hat{q},w}=\IN_w(g_i)\cdot \IN_w(h_i),
     \emath
     which finishes the proof.
   \end{proof}

   In order to be able to apply standard bases techniques we need to
   fix a $t$-local monomial ordering which refines a given weight
   vector $w$. 

   \begin{definition}\label{def:word}
     Fix any \emph{global} monomial ordering, say $>$, on $\Mon(\ux)$
     and  let $w=(w_0,\ldots,w_n)\in\R_{<0}\times\R^n$. 

     We define a $t$-local monomial ordering, say $>_w$, on
     $\Mon\big(t^\frac{1}{N},\ux\big)$ by
     \begin{displaymath}
       t^\frac{\alpha}{N}\cdot\ux^\beta\;>_w\;t^{\frac{\alpha'}{N}}\cdot\ux^{\beta'}       
     \end{displaymath}
     if and only if
     \begin{displaymath}
       w\cdot \left(\frac{\alpha}{N},\beta\right)>w\cdot \left(\frac{\alpha'}{N},\beta'\right)
     \end{displaymath}
     or
     \begin{displaymath}
       w\cdot \left(\frac{\alpha}{N},\beta\right)=w\cdot \left(\frac{\alpha'}{N},\beta'\right)
       \;\mbox{ and }\;
       \ux^\beta\;>\;\ux^{\beta'}.              
     \end{displaymath}
     Note that this ordering is indeed $t$-local since $w_0<0$, and
     that it depends on $w$ and on $>$, but assuming that $>$ is fixed
     we will refrain from writing $>_{w,>}$ instead of $>_w$.
   \end{definition}

   \begin{remark}\label{rem:restriction}
     If $N\;|\;M$ then
     $\Mon\big(t^\frac{1}{N},\ux\big)\subset\Mon\big(t^\frac{1}{M},\ux\big)$, as
     already mentioned. For $w\in \R_{<0}\times\R^n$ we may thus
     consider the ordering $>_w$ on both $\Mon\big(t^\frac{1}{N},\ux\big)$
     and on $\Mon\big(t^\frac{1}{M},\ux\big)$, and let us call them for a
     moment $>_{w,N}$ respectively $>_{w,M}$.  It is important to note,
     that the restriction of $>_{w,M}$ to $\Mon\big(t^\frac{1}{N},\ux\big)$
     coincides with $>_{w,N}$. We therefore omit the additional
     subscript in our notation.
     \medskip
     \begin{center}
     \framebox[12cm]{
       \begin{minipage}{11.7cm}
         \medskip
         \emph{We now fix some global monomial ordering $>$ on
           $\Mon(\ux)$, and given a vector $w\in\R_{<0}\times\R^n$ we will throughout this
         section always denote by $>_w$ the monomial ordering from
         Definition \ref{def:word}.}
         \medskip
       \end{minipage}
       }       
     \end{center}
     \medskip
   \end{remark}

   \begin{proposition}\label{prop:initialformunit}
     If $w\in\R_{<0}\times\R^n$ and $f\in R_N[\ux]$ with $\lt_{>_w}(f)=1$,
     then $\IN_w(f)=1$.     
   \end{proposition}
   \begin{proof}
     Suppose this is not the case then there exists a monomial of $f$,
     say $1\not=t^\alpha\cdot \ux^\beta\in\mathcal{M}_f$, such that 
     \bmath
       w\cdot (\alpha,\beta)\geq w\cdot (0,\ldots,0)=0,
     \emath
     and since $\lm_{>_w}(f)=1$ we must necessarily have equality. But
     since $>$ is global $\ux^\beta> 1$, which  implies that also
     $t^\alpha\cdot \ux^\beta>_w 1$, in contradiction to
     $\lm_{>_w}(f)=1$. 
   \end{proof}

   \begin{proposition}\label{prop:stdin}
     Let $w\in \R_{<0}\times\R^n$, $I\unlhd
     R_N[\ux]$ be an ideal, and let $G=\{g_1,\ldots,g_k\}$ 
     be a standard basis of $I$ with respect to $>_w$ then
     \begin{displaymath}
       \IN_w(I)=\big\langle \IN_w(g_1),\ldots,\IN_w(g_k)\big\rangle\unlhd K\big[t^\frac{1}{N},\ux\big],
     \end{displaymath}
     and in particular,
     \begin{displaymath}
       \tin_w(I)=\big\langle \tin_w(g_1),\ldots,\tin_w(g_k)\big\rangle\unlhd K[\ux].
     \end{displaymath}
   \end{proposition}
   \begin{proof}
     If $G$ is standard basis of $I$ then by Corollary
     \ref{cor:buchbergercriterion}
     every element $f\in I$ has a
     weak standard representation of the form
     \bmath
       u\cdot f=q_1\cdot g_1+\ldots +q_k\cdot g_k,
     \emath
     where $\lt_{>_w}(u)=1$ and 
     \bmath
       \lm_{>_w}(u\cdot f)\geq \lm_{>_w}(q_i\cdot g_i).
     \emath
     The latter in particular implies that
     \begin{displaymath}
       \ord_w(u\cdot f)=\deg_w\big(\lm_{>_w}(u\cdot f)\big)
       \geq \deg_w\big(\lm_{>_w}(q_i\cdot g_i)\big)
       =\ord_w(q_i\cdot g_i).
     \end{displaymath}
     We conclude therefore by Lemma \ref{lem:initialform}
     and Proposition \ref{prop:initialformunit} that
     \begin{displaymath}
       \IN_w(f)=\IN_w(u\cdot f)\in \big\langle \IN_w(g_1),\ldots,\IN_w(g_k)\big\rangle.
     \end{displaymath}

     For the part on the $t$-initial ideals just note that if $f\in
     I$ then by the above
     \bmath
       \IN_w(f)=\sum_{i=1}^k h_i\cdot \IN_w(g_i)
     \emath
     for some $h_i\in K\big[t^\frac{1}{N},\ux\big]$, and thus
     \begin{displaymath}
       \tin_w(f)= \sum_{i=1}^k h_i(1,\ux)\cdot \tin_w(g_i)\in 
       \langle\tin_w(g_1),\ldots,\tin_w(g_k)\rangle_{K[\ux]}.
     \end{displaymath}
   \end{proof}

   \begin{theorem}\label{thm:stdtin}
     Let $J\unlhd L[\ux]$ and $I\unlhd R_N[\ux]$ be ideals with $J=\langle
     I\rangle_{L[\ux]}$, let $w\in \R_{<0}\times\R^n$, 
     and let $G$ be a standard basis of $I$ with respect to $>_w$.
     \lang{

     }
     Then     
     \begin{displaymath}
       \tin_w(J)=\tin_w(I)=\big\langle \tin_w(G)\big\rangle\lhd K[\ux].
     \end{displaymath}
   \end{theorem}
   \begin{proof}
     Since $R_N[\ux]$ is noetherian, we may add a finite number of
     elements of $I$ to $G$ so as to assume that
     $G=(g_1,\ldots,g_k)$
     generates $I$. Since by Proposition \ref{prop:stdin} we already
     know that the $t$-initial forms of any standard basis of $I$ with
     respect to $>_w$ generate 
     $\tin_w(I)$ this does not change the right hand side. But then by
     assumption 
     \bmath
       J=\langle G\rangle_{L[\ux]},
     \emath
     and given an element $f\in J$ we can write it as
     \begin{displaymath}
       f=\sum_{i=1}^k t^\frac{-\alpha}{N\cdot M}\cdot a_i\cdot g_i
     \end{displaymath}
     for some $M>>0$, $a_i\in R_{N\cdot M}$ and $\alpha\in\N$.
     It follows that
     \begin{displaymath}
       t^\frac{\alpha}{N\cdot M}\cdot f=\sum_{i=1}^k a_i\cdot g_i\in
       \langle G\rangle_{R_{N\cdot M}[\ux]}.
     \end{displaymath}
     Since $G$ is a standard basis over $R_N[\ux]$ with respect to
     $>_w$ on $\Mon\big(t^\frac{1}{N},\ux\big)$ by Buchberger's
     Criterion \ref{cor:buchbergercriterion} $\spoly(g_i,g_j)$, $i<j$, has a
     weak standard representation
     \bmath
       u_{ij}\cdot \spoly(g_i,g_j)=
       \sum_{\nu=1}^kq_{ij\nu}\cdot g_\nu
     \emath
     with $u_{ij},q_{ij\nu}\in R_N[\ux]\subseteq R_{N\cdot M}[\ux]$
     and $\lt_{>_w}(u_{ij})=1$. Taking Remark \ref{rem:restriction}
     into account these are also weak standard representations with
     respect to the corresponding monomial ordering $>_w$ on
     $\Mon(t^\frac{1}{N\cdot M},\ux)$, and again by Buchberger's
     Criterion \ref{cor:buchbergercriterion} there exists a weak standard
     representation 
     \bmath
       u\cdot t^\frac{\alpha}{N\cdot M}\cdot f=
       \sum_{i=1}^k q_i\cdot g_i.
     \emath
     By Propositions \ref{lem:initialform} and
     \ref{prop:initialformunit} this implies that 
     \begin{displaymath}
       t^\frac{\alpha}{N\cdot M}\cdot\IN_w(f)=\IN_w\big(u\cdot
       t^\frac{\alpha}{N\cdot M}\cdot f\big)\in \big\langle\IN_w(G)\big\rangle. 
     \end{displaymath}
     Setting $t=1$ we get 
     \bmath
       \tin_w(f)=\big(t^\frac{k}{N\cdot M}\cdot\IN_w(f)\big)_{|t=1}\in
       \big\langle\tin_w(G)\big\rangle.
     \emath
   \end{proof}

   \begin{corollary}\label{cor:stdtin}
     Let $J=\langle I'\rangle_{L[\ux]}$
     with $I'\unlhd K\big[t^\frac{1}{N},\ux\big]$, 
     $w\in \R_{<0}\times\R^n$ and $G$ is a standard basis of
     $I'$ with respect to $>_w$ on 
     $\Mon\big(t^\frac{1}{N},\ux\big)$, then
     \begin{displaymath}
       \tin_w(J)=\tin_w(I')=\big\langle\tin_w(G)\big\rangle\unlhd K[\ux].
     \end{displaymath}
   \end{corollary}
   \begin{proof}
     Enlarge $G$ to a finite generating set $G'$ of $I'$, then $G'$ is still
     a standard basis of $I'$. 
     By Corollary \ref{cor:polynomialcase} $G'$ is then also a standard
     basis of 
     \begin{displaymath}
       I:=\langle G'\rangle_{R_N[\ux]}=\langle f_1,\ldots,f_k\rangle_{R_N[\ux]},
     \end{displaymath}
     and Theorem \ref{thm:stdtin} applied to $I$ thus shows that
     \begin{displaymath}
       \tin(J)=\big\langle\tin_w(G')\big\rangle.
     \end{displaymath}
     However, if $f\in G'\subset I'$ is one of the additional elements then it
     has a weak standard representation 
     \begin{displaymath}
       u\cdot f=\sum_{g\in G}q_g\cdot g
     \end{displaymath}
     with respect to $G$ and $>_w$, since $G$ is a standard basis of
     $I'$. Applying Propositions \ref{lem:initialform} and
     \ref{prop:initialformunit} then shows that
     $\IN_w(f)\in\langle\IN_w(G)\rangle$, which finishes the proof.
   \end{proof}

   \begin{remark}\label{rem:saturation}
     Note that if $I\unlhd R_N[\ux]$ and $J=\langle I\rangle_{L[\ux]}$,
     then
     \begin{displaymath}
       J\cap R_N[\ux]=I:\big\langle t^\frac{1}{N}\big\rangle^\infty,
     \end{displaymath}
     but the saturation is in general necessary.

     Since $L_N\subset L$ is a field extension Corollary
     \ref{cor:I=I^e} implies     
     \bmath
       J\cap L_N[\ux]=\langle I\rangle_{L_N[\ux]},
     \emath
     and it suffices to see that 
     \begin{displaymath}
       \langle I\rangle_{L_N[\ux]}\cap R_N[\ux]=I:\big\langle t^\frac{1}{N}\big\rangle^\infty.
     \end{displaymath}
     If $I\cap S_N\not=\emptyset$ then both sides of the equation
     coincide with $R_N[\ux]$, so that we may assume that $I\cap S_N$ is
     empty. Recall that $L_N=S_N^{-1}R_N$, so that if $f\in R_N[\ux]$
     with $t^\frac{\alpha}{N}\cdot f\in I$ for some $\alpha$, then
     \begin{displaymath}
       f=\frac{t^\frac{\alpha}{N}\cdot f}{t^\frac{\alpha}{N}}\in \langle I\rangle_{L_N[\ux]}\cap R_N[\ux].
     \end{displaymath}
     Conversely, if 
     \bmath
       f=\frac{g}{t^\frac{k}{N}}\in \langle I\rangle_{L_N[\ux]}\cap R_N[\ux]
     \emath
     with $g\in I$, then $g=t^\frac{\alpha}{N}\cdot f\in I$ and thus $f$ is
     in the right hand side.     
   \end{remark}

   \lang{
   \begin{lemma}\label{lem:monomialideal}
     Let $F\subset F'$ be a field extension, and $I=\langle
     \ux^{\alpha_1},\ldots, \ux^{\alpha_k}\rangle\unlhd F[\ux]$ be a
     monomial ideal.
     Then $I=\langle I\rangle_{F'[\ux]}\cap F[\ux]$.
   \end{lemma}
   \begin{proof}
     It suffices to show that $\langle I\rangle_{F'[\ux]}\cap F[\ux]\subseteq I$. For this we
     consider an $f\in \langle I\rangle_{F'[\ux]}\cap F[\ux]$. Since $\langle I\rangle_{F'[\ux]}=\langle
     \ux^{\alpha_1},\ldots, \ux^{\alpha_k}\rangle_{F'[\ux]}$ is a
     monomial ideal, for every term, say $f_j$, of $f$ there is some $i$ such that
     $\ux^{\alpha_i}\;|\;f_j$, i.e.\ $f_j=\ux^{\alpha_i}\cdot f'_j$ with
     $f'_j\in F'[\ux]$. However, since all coefficients of $f$ are in
     $F$ so must be all coefficients of $f'_j$, and thus
     $f_j=\ux^{\alpha_i}\cdot f'_j\in I$, which implies $f\in I$.
   \end{proof}

   \begin{lemma}\label{lem:groebnerbasis}
     Let $F\subset F'$ be a field extension, $f_1,\ldots,f_k\in
     F[\ux]$,  
     and $>$ a global monomial ordering on $\Mon(\ux)$. 
     Then every Gr\"obner basis $G$ of
     $\langle f_1,\ldots,f_k\rangle_{F[\ux]}$ with respect to $>$ is
     also a Gr\"obner basis of $\langle
     f_1,\ldots,f_k\rangle_{F'[\ux]}$. 
   \end{lemma}
   \begin{proof}
     If $G=\{g_1,\ldots,g_l\}$ then $g_i\in \langle
     f_1,\ldots,f_k\rangle_{F[\ux]}\subseteq \langle
     f_1,\ldots,f_k\rangle_{F'[\ux]}$, and
     \begin{displaymath}
       \langle f_1,\ldots,f_k\rangle_{F[\ux]}=\langle G\rangle_{F[\ux]}
     \end{displaymath}
     shows that
     \begin{displaymath}
       \langle f_1,\ldots,f_k\rangle_{F'[\ux]}=\langle G\rangle_{F'[\ux]}.
     \end{displaymath}
     If $s_{i,j}$ denotes the
     S-polynomial of $g_i$ and $g_j$, then by Buchberger's Criterion
     \ref{thm:buchbergercriterion} there exists a standard representation
     \begin{displaymath}
       s_{i,j}=q_{1,i,j}\cdot g_1+\ldots +q_{l,i,j}\cdot g_l
     \end{displaymath}
     with $q_{s,i,j}\in F[\ux]\subseteq F'[\ux]$. But
     then these same representations together with Buchberger's
     Criterion imply that $G$ is a Gr\"obner basis of $\langle
     f_1,\ldots,f_k\rangle_{F'[\ux]}$.
   \end{proof}
   }

   \begin{corollary}\label{cor:I=I^e}
     Let $F\subset F'$ be a field extension and $I\unlhd F[\ux]$.
     Then $I=\langle I\rangle_{F'[\ux]}\cap F[\ux]$.
   \end{corollary}
   \begin{proof}
     \kurz{The result is obvious if $I$ is generated by monomials. For
     the general case fix}\lang{Fix} any global monomial ordering $>$ on $\Mon(\ux)$
     and set $I^e=\langle I\rangle_{F'[\ux]}$.
     Since $I\subseteq I^e\cap F[\ux]\subseteq I^e$ we also have 
     \begin{equation}\label{eq:I=I^e:1}
       L_>(I)\subseteq L_>\big(I^e\cap F[\ux]\big)\subseteq L_>(I^e)\cap F[\ux].
     \end{equation}
     If we choose a standard basis
     $G=(g_1,\ldots,g_k)$ of $I$, then by 
     \lang{Lemma \ref{lem:groebnerbasis}}\kurz{Buchberger's Criterion}
     $G$ is also a Gr\"obner basis of $I^e$ and thus
     \begin{displaymath}
       L_>(I)=\langle \lm_>(g_i)\;|\;i=1,\ldots,k\rangle_{F[\ux]}
     \end{displaymath}
     and 
     \begin{displaymath}
       L_>(I^e)=\langle \lm_>(g_i)\;|\;i=1,\ldots,k\rangle_{F'[\ux]}
       =\big\langle L_>(I)\big\rangle_{F'[\ux]}.
     \end{displaymath}
     Since the latter is a monomial ideal, \lang{by Lemma
     \ref{lem:monomialideal}} we have
     \begin{displaymath}
       L_>(I^e)\cap F[\ux]=L_>(I).
     \end{displaymath}
     In view of \eqref{eq:I=I^e:1} this shows that
     \begin{displaymath}
       L_>(I)=L_>\big(I^e\cap F[\ux]\big),
     \end{displaymath}
     and since $I\subseteq I^e\cap F[\ux]$ this finishes the proof
     by Proposition \ref{prop:basicstd}.
   \end{proof}

   We can actually show more, namely, that for each $I\unlhd R_N[\ux]$
   and each $M>0$ (see Corollary \ref{cor:I^e})
   \begin{displaymath}
     \langle I\rangle_{R_{M\cdot N}[\ux]}\cap R_N[\ux]=I,
   \end{displaymath}
   and if $I$ is saturated with respect to $t^\frac{1}{N}$ then (see
   Corollary \ref{cor:initialideals2}) 
   \begin{displaymath}
     \IN_w\big(\langle I\rangle_{R_{M\cdot N}[\ux]}\big)=
     \big\langle\IN_w(G)\big\rangle,
   \end{displaymath}
   if $G$ is a standard basis of $I$ with respect to $>_w$. 
   \lang{This
   requires, however, some extra work which is partly outsourced to
   Section \ref{sec:properties}.}

   \kurz{
     For this we need the following simple observation.

   \begin{lemma}\label{cor:free}
     $R_{N\cdot M}[\ux]$ is a
     free $R_N[\ux]$-module with basis $\big\{1,t^{\frac{1}{N\cdot
         M}},\ldots,t^{\frac{M-1}{N\cdot M}}\big\}$.     
   \end{lemma}
   }

   \begin{corollary}\label{cor:I^e}
     If $I\unlhd R_N[\ux]$ then $\langle I\rangle_{R_{N\cdot M}[\ux]}\cap R_N[\ux]=I$.
   \end{corollary}
   \begin{proof}
     If $f=g\cdot h\in \langle I\rangle_{R_{N\cdot M}[\ux]}\cap R_N[\ux]$ with $g\in I$ and
     $h\in R_{N\cdot M}[\ux]$ then by \lang{Corollary}\kurz{Lemma} \ref{cor:free} there are
     uniquely determined $h_i\in R_N$ such that
     \bmath
       h=\sum_{i=0}^{M-1} h_i\cdot t^\frac{i}{N\cdot M},
     \emath
     and hence
     \bmath
       f=\sum_{i=0}^{M-1} (g\cdot h_i)\cdot t^\frac{i}{N\cdot M}
     \emath
     with $g\cdot h_i\in R_N[\ux]$. By assumption $f\in
     R_N[\ux]=R_{N\cdot M}[\ux]\cap\langle 1\rangle_{R_N[\ux]}$ and by
     \lang{Corollary}\kurz{Lemma} \ref{cor:free} we thus have 
     \bmath
       g\cdot h_i=0\lang{\;\;\;}\mbox{ for all }\lang{\;} i=1,\ldots,M-1.
     \emath
     But then $f=g\cdot h_0\in I$.
   \end{proof}   

   \begin{lemma}\label{lem:saturation}
     Let $I\unlhd R_N[\ux]$ be an ideal such that $I=I:\big\langle
     t^\frac{1}{N}\big\rangle^\infty$, then for any $M\geq 1$
     \begin{displaymath}
       \langle I\rangle_{R_{N\cdot M}[\ux]}=\langle I\rangle_{R_{N\cdot M}[\ux]}:\big\langle
       t^\frac{1}{N\cdot M}\big\rangle^\infty.
     \end{displaymath}
   \end{lemma}
   \begin{proof}
     Let $f,h\in R_{N\cdot M}[\ux]$, $\alpha\in\N$, $g\in I$ such that
     \begin{equation}\label{eq:saturation:1}
       t^\frac{\alpha}{N\cdot M}\cdot f=g\cdot h.
     \end{equation}
     We have to show that $f\in \langle I\rangle_{R_{N\cdot
         M}[\ux]}$. For this purpose do  division
     with remainder in order to get 
     \bmath
       \alpha=a\cdot M+b\;\;\;\mbox{ with }\;\; 0\leq b<M.
     \emath
     By  \lang{Corollary}\kurz{Lemma} \ref{cor:free} there are  $h_i, f_i\in R_N[\ux]$
     such that 
     $f=\sum_{i=0}^{M-1}f_i\cdot t^\frac{i}{N\cdot M}$ and 
     $h=\sum_{i=0}^{M-1}h_i\cdot t^\frac{i}{N\cdot
       M}$. \eqref{eq:saturation:1} then translates into
     \begin{displaymath}
       \sum_{i=0}^{M-1-b}t^\frac{b+i}{N\cdot M}\cdot t^\frac{a}{N}\cdot f_i 
       +\sum_{i=M-b}^{M-1}t^\frac{b+i-M}{N\cdot M}\cdot t^\frac{a+1}{N}\cdot f_i 
       =
       \sum_{i=0}^{M-1}g\cdot h_i\cdot t^\frac{i}{N\cdot M},
     \end{displaymath}
     and since $\big\{1,t^\frac{1}{N\cdot M},\ldots,t^\frac{M-1}{N\cdot
       M}\big\}$  is $R_N[\ux]$-linearly independent we can compare
     coefficients to find
     \bmath
       t^\frac{a}{N}\cdot f_i=g\cdot h_{b+i}\in I
     \emath
     for $i=0,\ldots,M-b-1$, and 
     \bmath
       t^\frac{a+1}{N}\cdot f_i=g\cdot h_{b+i-M}\in I
     \emath
     for $i=M-b,\ldots,M-1$. In any case, since $I$ is saturated with
     respect to $t^\frac{1}{N}$ by assumption we conclude that $f_i\in
     I$ for all $i=0,\ldots,M-1$, and therefore $f\in \langle I\rangle_{R_{N\cdot
         M}[\ux]}$.     
   \end{proof}

   \begin{corollary}\label{cor:initialideals}
     Let $J\unlhd L[\ux]$ be an ideal such that $J=\langle J\cap
     R_N[\ux]\rangle_{L[\ux]}$, let $w\in \R_{<0}\times\R^n$, and let
     $G$ be a standard  basis of 
     $J\cap R_N[\ux]$ with respect to $>_w$. 

     Then for all $M\geq 1$
     \begin{displaymath}
       \IN_w\big(J\cap R_{N\cdot
         M}[\ux]\big)=\big\langle\IN_w(G)\big\rangle\lhd
       K\big[t^\frac{1}{N\cdot M},\ux\big]
     \end{displaymath}
     and
     \begin{displaymath}
       \tin_w\big(J\cap R_{N\cdot
         M}[\ux]\big)=\big\langle\tin_w(G)\big\rangle=\tin_w\big(J\cap
       R_N[\ux]\big)\lhd K[\ux]. 
     \end{displaymath}
   \end{corollary}
   \begin{proof}
     Enlarge $G$ to a generating set $G'$ of $I=J\cap R_N[\ux]$ over $R_N[\ux]$ by
     adding a finite number of elements of $I$. Then
     \begin{displaymath}
       \big\langle L_{>_w}(G')\big\rangle\subseteq
       \big\langle L_{>_w}(I)\big\rangle=
       \big\langle L_{>_w}(G)\big\rangle\subseteq
       \big\langle L_{>_w}(G')\big\rangle
     \end{displaymath}
     shows that $G'$ is still a standard basis of $I$ with respect to
     $>_w$. 
     So we can assume that $G=G'$.

     By Proposition \ref{prop:stdin} it suffices to show that $G$ is
     also a standard basis of $J\cap R_{N\cdot M}[\ux]$. Since by assumption
     \bmath
       J=\langle I\rangle_{L[\ux]}=\langle G\rangle_{L[\ux]},
     \emath
     Corollary \ref{cor:I=I^e} implies that
     \begin{displaymath}
       J\cap L_{N\cdot M}[\ux]=\langle G\rangle_{L_{N\cdot M}[\ux]}
       =S_{N\cdot M}^{-1}\langle G\rangle_{R_{N\cdot M}[\ux]}.
     \end{displaymath}
     Moreover, by Remark \ref{rem:saturation} the ideal $I=\langle G\rangle_{R_N[\ux]}$ is
     saturated with respect to $t^\frac{1}{N}$ and by Lemma
     \ref{lem:saturation} therefore also $\langle G\rangle_{R_{N\cdot
         M}[\ux]}$ is saturated with respect to $t^\frac{1}{N\cdot M}$,
     which implies that
     \begin{displaymath}
       J\cap R_{N\cdot M}[\ux]=
       S_{N\cdot M}^{-1}\langle G\rangle_{R_{N\cdot M}[\ux]}\cap R_{N\cdot M}[\ux]
       =\langle G\rangle_{R_{N\cdot M}[\ux]}.
     \end{displaymath}
     Since $G=(g_1,\ldots,g_k)$ is a standard basis of $I$
     every $\spoly(g_i,g_j)$, $i<j$, has a weak standard
     representation with respect to $G$ and $>_w$ over $R_N[\ux]$
     by Buchberger's Criterion \ref{cor:buchbergercriterion}, and
     these are of course also weak standard 
     representations over $R_{N\cdot M}[\ux]$, so that again by
     Buchberger's Criterion $G$ is a standard basis of $\langle
     G\rangle_{R_{N\cdot M}[\ux]}=J\cap R_{N\cdot M}[\ux]$. 
   \end{proof}

   \begin{corollary}\label{cor:initialideals2}
     Let $I\unlhd R_N[\ux]$ be an ideal such that $I=I:\big\langle
     t^\frac{1}{N}\big\rangle^\infty$, let $w\in \R_{<0}\times\R^n$, and let
     $G$ be a standard  basis of $I$ with respect to $>_w$. 

     Then for all $M\geq 1$
     \begin{displaymath}
       \IN_w\big(\langle
       I\rangle_{R_{N\cdot M}[\ux]}\big)=
       \big\langle\IN_w(G)\big\rangle
       \lhd K\big[t^\frac{1}{N\cdot M},\ux\big]
     \end{displaymath}
     and 
     \begin{displaymath}
       \tin_w\big(\langle
       I\rangle_{R_{N\cdot M}[\ux]}\big)=
       \big\langle\tin_w(G)\big\rangle=\tin_w(I)
       \lhd K[\ux]
     \end{displaymath}
   \end{corollary}
   \begin{proof}
     If we consider $J=\langle I\rangle_{L[\ux]}$ then by Remark
     \ref{rem:saturation} $J\cap R_N[\ux]=I$, and moreover, by Lemma
     \ref{lem:saturation} also $\langle I\rangle_{R_{N\cdot M}[\ux]}$
     is saturated with respect to $t^\frac{1}{N\cdot M}$, so that
     applying Remark \ref{rem:saturation} once again we also find
     $J\cap R_{N\cdot M}[\ux]=\langle I\rangle_{R_{N\cdot M}[\ux]}$. 
     The result therefore follows
     from Corollary \ref{cor:initialideals}.
   \end{proof}

   \begin{corollary}\label{cor:initialideals3}
     Let $J\unlhd L[\ux]$ be an ideal such that $J=\langle J\cap
     R_N[\ux]\rangle_{L[\ux]}$, let $w=(-1,0,\ldots,0)$ and let $M\geq 1$. Then
     \begin{displaymath}
       1\in\IN_\omega\big(J\cap R_N[\ux]\big)
       \;\;\;\Longleftrightarrow\;\;\;
       1\in\IN_\omega\big(J\cap R_{N\cdot M}[\ux]\big).
     \end{displaymath}
   \end{corollary}
   \begin{proof}
     Suppose that $f\in J\cap R_{N\cdot M}[\ux]$ with
     $\IN_\omega(f)=1$, and let $G=(g_1,\ldots,g_k)$ be standard basis
     of $J\cap R_N[\ux]$ with respect to $>_w$. By Corollary
     \ref{cor:initialideals} 
     \begin{displaymath}
       1=\IN_\omega(f)\in\big\langle\IN_\omega(g_1),\ldots,\IN_\omega(g_k)\big\rangle
       \lhd K\big[t^\frac{1}{N\cdot M},\ux\big],
     \end{displaymath}
     and since this ideal and $1$ are $w$-quasihomogeneous, there
     exist $w$-quasihomogeneous elements $h_1,\ldots,h_k\in
     K\big[t^\frac{1}{N\cdot M},\ux\big]$ such that
     \begin{displaymath}
       1=\sum_{i=1}^k h_i\cdot \IN_\omega(g_i),
     \end{displaymath}
     where each summand on the right hand side (possibly zero) is
     $w$-quasihomogeneous of $w$-degree zero. Since
     $w=(-1,0,\ldots,0)$ this forces $h_i\in K[\ux]$ for all
     $i=1,\ldots,k$ and thus $1\in \IN_\omega(J\cap R_N[\ux])$. The
     converse is clear anyhow. 
   \end{proof}

   We want to conclude the section by a remark on the saturation.

   \begin{proposition}\label{prop:saturation}
     If $f_1,\ldots,f_k\in K[t,\ux]$ and $I=\langle
     f_1,\ldots,f_k\rangle\unlhd K[t]_{\langle t\rangle}[\ux]$ then
     \begin{displaymath}
       \langle I\rangle_{R_1[\ux]}:
       \langle t\rangle^\infty
       =
       \big\langle I:\langle  t\rangle^\infty
       \big\rangle_{R_1[\ux]}.
     \end{displaymath}
   \end{proposition}
   \begin{proof}
     Let $>_1$ be any global monomial ordering on $\Mon(\ux)$ and
     define a $t$-local monomial ordering on $\Mon\big(t,\ux)$ by
     \begin{displaymath}
       t^\alpha\cdot \ux^\beta\;>\;t^{\alpha'}\cdot \ux^{\beta'}       
     \end{displaymath}
     if and only if
     \begin{displaymath}
       \ux^\alpha\;>_1\;\ux^{\alpha'}
       \;\mbox{ or }\;
       \big(\ux^\alpha=\ux^{\alpha'}\;\mbox{ and }\;\alpha<\alpha'\big).
     \end{displaymath}
     Then 
     \begin{displaymath}
       \{f\in R_1[\ux]\;|\;\lt_>(f)=1\}=\{1+t\cdot p\;|\;p\in K[t]\},
     \end{displaymath}
     and thus
     \begin{displaymath}
       R_1[\ux]_>=R_1[\ux]\;\;\;\mbox{ and }\;\;\;
       K[t,\ux]_>=K[t]_{\langle t\rangle}[\ux].
     \end{displaymath}
     Using \lang{Algorithm \ref{alg:saturation}}\kurz{Remark \ref{rem:polynomialcase}} we
     can compute at the same time a standard basis of $\langle
     I\rangle_{R_1[\ux]}:\langle t\rangle^\infty$ and of $\langle
     I\rangle_{K[t]_{\langle t\rangle}[\ux]}:\langle t\rangle^\infty$ 
     with respect to $>$. Since a standard basis is a generating set
     in the localised ring the result follows.
   \end{proof}
   
   \lang{

   \section{Some Properties of $R_N[\ux]$, $L_N[\ux]$, and $L[\ux]$}\label{sec:properties}

   In this section we gather some straight forward but useful
   properties of the rings we are working with and their relations
   among each other.

   \begin{lemma}\label{lem:free}
     $R_{N\cdot M}$ is a finite
     free $R_N$-module with basis $\big\{1,t^{\frac{1}{N\cdot
         M}},\ldots,t^{\frac{M-1}{N\cdot M}}\big\}$.

     In particular, $R_{N\cdot M}$
     is integral over $R_N$.
   \end{lemma}
   \begin{proof}
     Note that $f=\sum_{i=0}^\infty a_i\cdot t^\frac{i}{N\cdot M}\in R_{N\cdot M}$
     can be written as
     \begin{displaymath}
       f=\sum_{j=0}^{M-1} t^\frac{j}{N\cdot M}\cdot \sum_{i=0}^\infty
       a_{j+i\cdot M}\cdot t^\frac{i}{N}\in 
       \left\langle 1,t^{\frac{1}{N\cdot
         M}},\ldots,t^{\frac{M-1}{N\cdot M}}\right\rangle_{R_N}.
     \end{displaymath}
     Moreover, since no terms can cancel $f=0$ if and only if 
     \begin{displaymath}
       \sum_{i=0}^\infty
       a_{j+i\cdot M}\cdot t^\frac{i}{N}=0
     \end{displaymath}
     for all $j=0,\ldots,M-1$. Thus $R_{N\cdot M}$ is free over $R_N$
     with basis $\big\{1,t^{\frac{1}{N\cdot
         M}},\ldots,t^{\frac{M-1}{N\cdot M}}\big\}$.
   \end{proof}

   \begin{corollary}\label{cor:free}
     $R_{N\cdot M}[\ux]$ is a
     free $R_N[\ux]$-module with basis $\big\{1,t^{\frac{1}{N\cdot
         M}},\ldots,t^{\frac{M-1}{N\cdot M}}\big\}$.     
   \end{corollary}
   \begin{proof}
     If $f=\sum_{|\alpha|=0}^d a_{\alpha}\cdot \ux^\alpha\in R_{N\cdot
       M}[\ux]$ with $a_{\alpha}\in R_{N\cdot M}$ is given, then by
     Lemma \ref{lem:free} there
     exist $a_{\alpha,i}\in R_N$ such that
     \begin{displaymath}
       a_\alpha=\sum_{i=0}^{M-1}a_{\alpha,i}\cdot t^\frac{i}{N\cdot M},
     \end{displaymath}
     and thus 
     \begin{displaymath}
       f=\sum_{i=0}^{M-1}t^\frac{i}{N\cdot M}\cdot \sum_{|\alpha|=0}^d
       a_{\alpha,i}\cdot \ux^\alpha\in 
       \left\langle 1,t^{\frac{1}{N\cdot
         M}},\ldots,t^{\frac{M-1}{N\cdot M}}\right\rangle_{R_N[\ux]}.
     \end{displaymath}
 
     Suppose now that 
     \begin{displaymath}
       \sum_{i=0}^{M-1}t^\frac{i}{N\cdot M}\cdot f_i=0
     \end{displaymath}
     with 
     \begin{displaymath}
       f_i=\sum_{|\alpha|=0}^d a_{\alpha,i}\cdot \ux^\alpha\in R_N[\ux].
     \end{displaymath}
     Then
     \begin{displaymath}
       0=\sum_{|\alpha|=0}^d\ux^\alpha\cdot
       \sum_{i=0}^{M-1}t^\frac{i}{N\cdot M}\cdot a_{\alpha,i},
     \end{displaymath}
     and since the $\ux^\alpha$ are linearly independent over $R_N$ it
     follows that
     \begin{displaymath}
       \sum_{i=0}^{M-1}t^\frac{i}{N\cdot M}\cdot a_{\alpha,i}=0 
     \end{displaymath}
     for all $|\alpha|\leq d$. But then by Lemma \ref{lem:free} we
     have $a_{\alpha,i}=0$ for all $i$ and and all $\alpha$, which
     implies that $f_i=0$ for $i=0,\ldots, M-1$. 
   \end{proof}

   \begin{corollary}\label{cor:algebraic}
     $L$ is algebraic over $L_N$.
   \end{corollary}
   \begin{proof}     
     If $A\subset B$ is an integral extension of integral domains,
     then $\Quot(A)\subset\Quot(B)$ is algebraic. For this consider 
     $0\not=b\in B$ and the integral relation $b^k+a_1\cdot
     b^{k-1}+\ldots +a_k=0$ with $a_i\in A$ which it fulfils by
     assumption. Then 
     \begin{displaymath}
       a_k\cdot \frac{1}{b^k}+\ldots+a_1\cdot\frac{1}{b}+1=0
     \end{displaymath}
     is an algebraic relation of $\frac{1}{b}$ over $\Quot(A)$. This
     shows that $\Quot(B)$ is algebraic over $\Quot(A)$, since every
     element of the former is of the form $\frac{b'}{b}$. In
     particular, $L_{N\cdot
       M}=\Quot(R_{N\cdot M})$ is algebraic over $L_N=\Quot(R_N)$.

     If $a\in L$, then there is an $M$ such that $a\in L_M\subseteq
     L_{N\cdot M}$, and
     therefore $a$ is algebraic over $L_N$. This shows that $L$ is
     algebraic over $L_N$. 
   \end{proof}

   \begin{corollary}\label{cor:integrality}
     $L[\ux]$ is integral over $L_N[\ux]$.
   \end{corollary}
   \begin{proof}
     If $A\subset B$ is an integral ring extension, then so is
     $A[x]\subset B[x]$. To see this let $f=\sum_{i=0}^kb_i\cdot x^i\in B[x]$
     be given. Then $b_i$ is integral, over $A$ and thus it is integral
     over $A[x]$. But since $x^i\in A[x]$ we have that 
     \begin{displaymath}
       f=\sum_{i=0}^kb_ix^i\in A[x][b_0,\ldots,b_k]
     \end{displaymath}
     is an element of the integral ring extension $A[x]\subset
     A[x][b_0,\ldots,b_k]$. This shows that $B[x]$ is integral over
     $A[x]$. The result follows thus from Corollary \ref{cor:algebraic}.
   \end{proof}

   \begin{corollary}\label{cor:going}
     The ring extension $L_N[\ux]\subset L[\ux]$ satisfies the
     lying-over, going-up and the going-down property.
   \end{corollary}
   \begin{proof}
     See \cite[Prop.\ 5.10, Thm.\ 5.11  and Thm.\ 5.16]{AM69}.
   \end{proof}

   \begin{corollary}\label{cor:dimLLN}
     Let $I\unlhd L[\ux]$ be an ideal then $L[\ux]/I$ is integral over
     $L_N[\ux]/I\cap L_N[\ux]$. 

     In particular, $\dim(I)=\dim(I\cap L_N[\ux])$.
   \end{corollary}
   \begin{proof}
     See \cite[Prop.\ 5.6]{AM69}  and \cite[Prop.\ 9.2]{Eis96}.
   \end{proof}

   }



\providecommand{\bysame}{\leavevmode\hbox to3em{\hrulefill}\thinspace}

\end{document}